\documentclass[10pt]{amsart}
\usepackage{amssymb}
\usepackage{amsfonts}
\usepackage{amscd}
\usepackage[dvips]{graphicx}
\usepackage[all,cmtip]{xy}
\usepackage{hyperref}
\usepackage{yhmath}
\usepackage{ulem}
\usepackage[all]{xy}
\usepackage{mathdots}
\usepackage{leftidx}
\usepackage{mathrsfs}
\usepackage{amsmath,amssymb, amsthm}
\usepackage{color}
\usepackage{xypic,color}
\usepackage{tikz}
\usepackage{mathdots}
\usepackage{picture}
\usepackage{enumerate}
\usepackage{makecell}
\usepackage{extarrows}
\usepackage{graphicx}
\usepackage{caption}
\usepackage{subcaption}
\usepackage{float}
\numberwithin{equation}{section}
\newtheorem{Them}{Theorem}[section]
\newtheorem{Lem}[Them]{Lemma}
\newtheorem{Def}[Them]{Definition}
\newtheorem{Cor}[Them]{Corollary}
\newtheorem{Prop}[Them]{Proposition}
\newtheorem{Ex}[Them]{Example}

\newcommand{\old}[1]{{\color{red}#1}}

\newcommand{\new}[1]{{\color{blue}#1}}

\newcommand{\ind}{{\mathsf{ind}}}

\newcommand{\rad}{{\mathsf{rad}}}

\newcommand{\D}{{\mathsf{D}}}
\newcommand{\RS}{{\mathsf{Rsupp}}}
\newcommand{\LS}{{\mathsf{Lsupp}}}

\newcommand{\m}{\mathsf{mod}}
\newcommand{\stmod}{ \mathsf{\underline{mod}}}

\newcommand{\Hom}{{\mathsf {Hom}}}

\newcommand{\StHom}{\mathsf{\underline{Hom}}}

\title[A construction of simple-minded systems over domestic Brauer graph algebras]{A construction of simple-minded systems over domestic Brauer graph algebras II: the 1-domestic case}
\author{Zhen Zhang}

\address{Zhen Zhang
	\newline Faculty of Arts and Sciences 
	\newline Beijing Normal University 
	\newline  Zhuhai 519087
	\newline P.R.China}
\email{zhangzhen@bnu.edu.cn}

\date{version of \today}

\newcommand{\sdp}{\times\kern-.2em\vrule height1.1ex depth-.05ex}

 \normalbaselines
\thanks{The research work is supported  by NSFC ( No.12301044).}
\begin{document}
\renewcommand{\thefootnote}{\alph{footnote}}
\setcounter{footnote}{-1} \footnote{\it{Mathematics Subject Classification(2020)}: 16G70, 18G65.}
\renewcommand{\thefootnote}{\alph{footnote}}
\setcounter{footnote}{-1}
\footnote{ \it{Keywords}: 1-domestic Brauer graph algebra; simple-minded system;  covering theory; maximal orthogonal system;  Euclidean component.}
\setcounter{footnote}{-1}

\begin{abstract}
Let $A$  be a 1-domestic Brauer graph algebra.  By covering theory and the characterization of simple-minded systems of 2-domestic Brauer graph algebras, we construct a family  of objects in $A$-$\stmod$ to be a simple-minded system and our construction provides all simple-minded systems in $A$-$\stmod$.  
\end{abstract}

\maketitle

\section{Introduction}
Simple-minded systems, introduced by Koenig and Liu \cite{KL}, is a family of objects that satisfies orthogonality and generating condition in the stable module category of any artin algebra. 
Let $C$ be a 2-domestic Brauer graph algebra. In a previous paper \cite{Z1}, We showed that a family $\mathcal{S}$ of objects in $C$-$\stmod$ is a simple-minded system  if and only if $\mathcal{S}$ is a maximal orthogonal system  containing at least one object for each stable  Euclidean component of the  stable AR-quiver $_{s}\Gamma_{C}$ of  $C$. we also present a full construction for a family  of objects  to be a simple-minded system in the stable module category of a 2-domestic Brauer graph algebra. 

According to Bocian and Skowronski \cite{BoS},  a Brauer graph algebra $A$ is 1-domestic if and only if the Brauer graph $G$ of $A$ is  a tree with multiplicity $m(i)=2$ for exactly two vertices $i=i_{1}, i_{2}$ and multiplicity $m(j)=1$ for any other vertex $j$, or a  graph with 
a unique cycle of odd length and the multiplicity of each vertices of $G$ is one.  They also showed that a domestic Brauer graph algebra is $1$-domestic or $2$-domestic, and there is no $n$-domestic Brauer graph algebras for $n\geq3.$ In this paper, by covering theory and the characterization of simple-minded systems of  $2$-domestic  Brauer graph algebras, we provide a construction of simple-minded systems in the stable module category of a $1$-domestic  Brauer graph algebra.  We state our main results as follows. 

\begin{Them}$($Theorem \ref{sms-1-2-domestic}$)$\label{sms-1-2-domestic-0}
Let $A$ be a $1$-domestic Brauer graph algebra and  $\overline{F}$ the dense covering functor from $C$-$\stmod$ to $A$-$\stmod$, where $C$ is a $2$-domestic Brauer graph algebra.  Then 
\begin{enumerate}[$(1)$]
\item If  $\mathcal{S}$ be a $\overline{\varphi}$-stable simple-minded system in $C$-$\stmod$, then $\overline{F}(\mathcal{S})$ is a simple-minded system in $A$-$\stmod$.
\item If  $\mathcal{M}$ be a simple-minded system in $A$-$\stmod$, then $\overline{F}^{-1}(\mathcal{M})$ is a $\overline{\varphi}$-stable simple-minded system in $C$-$\stmod$.
\end{enumerate}
\end{Them}

\begin{Cor}$($Corollary \ref{sms-BGA}$)$\label{sms-BGA-0}
Let $A$ be a 1-domestic Brauer graph algebra and $\mathcal{S}$ a maximal orthogonal system which contains at least one object for the Euclidean component. Then $\mathcal{S}$ is a simple-minded system in $A$-$\stmod$ and the cardinality of $\mathcal{S}$ is the number of non-projective simple $A$-modules.
\end{Cor}

Combining Theorem \ref{sms-2Domes-BGA} (\cite[Theorem 5.14]{Z1}) with Corollary \ref{sms-BGA-0},  we present a full  characterization of  simple-minded systems  in the stable module category of a domestic  Brauer graph algebra. We also provide a construction of simple-minded systems over domestic  Brauer graph algebras.

\begin{Cor}$($Corollary \ref{sms-BGA-1}$)$\label{sms-BGA-01}
	Let $A$ be a domestic Brauer graph algebra and $\mathcal{S}$  a family of objects in $A$-$\stmod$.
	Then $\mathcal{S}$ is a simple-minded system in $A$-$\stmod$ if and only if $\mathcal{S}$ is a maximal orthogonal system which contains at least one object for each Euclidean component.
\end{Cor}

 This paper is organized as follows. In Section 2, we recall simple-minded systems, repetitive algebras, Brauer graph algebras and some related concepts and conclusions. In Section 3, we characterize simple-minded systems by  covering theory  in $A$-$\stmod$ over a 1-domestic Brauer graph algebra $A$. In Section 4,  we  study the action of $\overline{\varphi}$ in the stable AR-quiver $_{s}\Gamma_{\widehat{B}}$ of  the repetitive algebra $\widehat{B}$ and determine stable bricks in $A$-$\stmod$.   In Section 5,  we  present a construction of simple-minded systems for  1-domestic Brauer graph algebras.

\section{Preliminary}	
We always assume that $k$ is an algebraically closed field. Let $A$ be a finite dimensional $k$-algebra. We denote by $A$-$\m$ the category of finite dimensional left $A$-modules and $A$-$\stmod$ the stable (projective) module  category of finite dimensional left $A$-modules. For two objects $X,Y$ in $A$-$\stmod$, the abelian group  $\StHom_R(X,Y)$ from $X$ to $Y$ is the quotient  $\Hom_A(X,Y)/\mathcal{P}(X,Y)$, where $\mathcal{P}(X,Y)$ is the subgroup of $\Hom_A(X,Y)$ consisting of all $A$-module homomorphisms  which factor through a projective  $A$-module.  We denote by $\Omega_{A}$ (resp. $\Omega^{-1}_{A}$) the  {\bf syzygy functor} (resp. {\bf cosyzygy functor}) which assigns to any object $M$ of $A$-$\stmod$ the kernel of its projective cover $P_{A}(M)\twoheadrightarrow M$ (resp. cokernel of its injective envelop $M\hookrightarrow I_{A}(M)$) in $A$-$\m$.

Let $A$ be a representation-infinite $k$-algebra and $\Gamma_{A}$ the AR-quiver of $A$. By a {\bf component} of $\Gamma_{A}$ we mean a connected component of $\Gamma_{A}$. By a {\bf quasi-tube},  we mean a translation quiver $\Gamma$ such that its full translation subquiver formed by all vertices which are not projective-injective is a tube (of the form $\mathbb{Z}A_{\infty}/(\tau^{r})$ for a positive integer $r$).  We say $\Gamma$ is a {\bf homogeneous tube} provided that  $\Gamma$ is of the form $\mathbb{Z}A_{\infty}/(\tau)$.  A component $\mathcal{C}$ is called an {\bf Euclidean component} if $\mathcal{C}$ is of the form $\mathbb{Z}\Delta$, where $\Delta$ is an   Euclidean quiver in  $\{\widetilde{A}_{n}(n\geq1),\widetilde{D}_{m}(m\geq4),\widetilde{E}_{6},\widetilde{E}_{7},\widetilde{E}_{8}\}$.  Recall that a connected component $\mathcal{C}$ of an AR-quiver is called  {\bf stable generalized standard}, if $\underline{\rad}^{\infty}(X,Y)=0$  for all $X$ and $Y$ in $\mathcal{C}$, where $\underline{\rad}^{\infty}(X,Y)=\rad^{\infty}(X,Y)/P(X,Y)$, $\rad^{\infty}(-,-)$ is  the intersection of $\rad^{t}(-,-)$ for $t\geq1$ and $P(X,Y)$ is the subspace of $\Hom_{A}(X,Y)$ which factors through a projective module. Moreover, a family of connected components  $\{\mathcal{C}_{i}\}_{i\in J}$ ($J$ an index set)  of $\Gamma_A$ is called  {\bf stable generalized standard}, if given  $X\in\mathcal{C}_{i}$ and  $Y\in\mathcal{C}_{j}$, we have $\underline{\rad^{\infty}}(X,Y)=0$.  An $A$-module $X$ is said to be {\bf $\tau$-periodic module} if there is a positive integer $m$ such that $X\cong\tau^{m}(X)$, otherwise, $X$ is called a {\bf non-$\tau$-periodic module}.

\subsection{Simple-minded system}
Let $\mathcal{T}$ be a  triangulated category with the shift functor $[1]$.   For two families $\mathcal{S}_{1}, \mathcal{S}_{2}$ of objects in $\mathcal{T}$, $\mathcal{S}_{1}\ast\mathcal{S}_{2}$ is defined to be a family of objects which consists of  the middle term in a triangle $S_{1} \longrightarrow  X \longrightarrow S_{2} \longrightarrow S_{1}[1]$ satisfying $\ S_{1}\in \mathcal{S}_{1}$ and $S_{2}\in \mathcal{S}_{2}.$   

We denotes $(\mathcal{S})_{0}=\{0\}$ and inductively defines $(\mathcal{S})_{n}=(\mathcal{S})_{n-1}\ast(\mathcal{S}\cup\{0\})$ for $n\in\mathbb{Z}^{+}$.  Note that $(\mathcal{S})_{n}\subseteq(\mathcal{S})_{n+1}$. We say that $\mathcal{S}$ is {\it extension-closed}, if $\mathcal{S}\ast\mathcal{S}\subseteq \mathcal{S}$. We denote the {\bf extension closure} of a family $\mathcal{S}$ of objects in $\mathcal{T}$ as \[\mathcal{F}(\mathcal{S}):=\bigcup_{n\geq0}(\mathcal{S})_{n}.\] 
It is known that  $\mathcal{F}(\mathcal{S})$ is the smallest extension-closed full subcategory of  $\mathcal{T}$  containing $\mathcal{S}$.

\begin{Def}\label{brick-orthogonal-system}
Let $\mathcal{T}$ be an additive $k$-category.  An object $M$ in $\mathcal{T}$ is a {\bf stable brick} if $\mathcal{T}(M,M)\cong k$.   Moreover, a family $\mathcal{M}$ of stable bricks in $\mathcal{T}$  is an {\bf orthogonal system} if $\mathcal{T}(M,N)=0$ for all distinct  $M, N$ in $\mathcal{M}$.
\end{Def}

\begin{Def}\label{definition-sms-right-tir} {\rm(\cite[Definition 2.1]{KL}, \cite[Definition 2.4, 2.5]{Dugas})} 
Let $\mathcal{T}$ be a  triangulated category.  A family of objects $\mathcal{S}$ in $\mathcal{T}$ is a {\bf simple-minded system} if the following  conditions are satisfied$\colon$
\begin{enumerate}[$(1)$]
\item {\rm(Orthogonality)} $\mathcal{S}$ is an orthogonal system in $\mathcal{T}$. 
\item {\rm(Generating condition)} Extension closure $\mathcal{F}(\mathcal{S})$ of $\mathcal{S}$ is equal to $\mathcal{T}$.
\end{enumerate}
\end{Def}

Koenig and Liu \cite{KL} introduced a weaker concept than simple-minded system, namely weakly simple-minded system.

\begin{Def}\label{definition-wsms-right-tir} {\rm(\cite[Definition 5.3]{KL})} 
Let $\mathcal{T}$ be a triangulated category.  A family of objects $\mathcal{S}$ in $\mathcal{T}$ is a {\bf weakly simple-minded system}  if the following two conditions are satisfied$\colon$
\begin{enumerate}[$(1)$]
\item {\rm(Orthogonality)} $\mathcal{S}$ is an orthogonal system in $\mathcal{T}$. 
\item {\rm(Weakly generating condition)} For any non-zero object $X$ in $\mathcal{T}$, there is an object $S$ in $\mathcal{S}$ such that $\mathcal{T}(S,X)\ncong 0.$
\end{enumerate}
\end{Def}

\begin{Them}$($\cite[Theorem 1.3]{CLZ}$)$ \label{simple-module-and-n-tube}
Let $A$ be a self-injective algebra and $\mathcal{C}$ a quasi-tube of rank $n$. Then the number of elements in a simple-minded system of $A$ lying in $\mathcal{C}$ is strictly less than $n$. In particular, none of the indecomposable modules in a simple-minded system  lie in the homogeneous tubes of the AR-quiver.
\end{Them}	

The following result provided us a sufficient and necessary condition for an orthogonal system to be a simple-minded system in the stable module category of a 2-domestic Brauer graph algebra.
\begin{Them}{\rm(\cite[Theorem 5.14]{Z1}}\label{sms-2Domes-BGA}
Let $A$ be a 2-domestic Brauer graph algebra and $\mathcal{S}$ a maximal orthogonal system  which contains at least one object for each Euclidean component. Then $\Omega^{-1}(\mathcal{S})\subseteq\mathcal{F}(\mathcal{S})$, in particular, $\mathcal{S}$ is a simple-minded system in $A$-$\stmod.$
\end{Them} 

\subsection{Repetitive algebra and domestic Brauer graph algebra}
In this subsection, we shall present the definition of Brauer graph algebras in terms of the 
repetitive algebra of a hereditary algebra of Euclidean type $\widetilde{A}_{m}$, instead of via combinatorial data of a Brauer graph directly. 

Let $B$ be a finite dimensional algebra.  The {\bf repetitive algebra} $\widehat{B}$ of  $B$  is defined to be \[\widehat B=
\bigoplus_{i\in{\mathbb{Z}}}B_{i}\oplus Q_{i}, 
\]
where $B_{i}$ (resp. $ Q_{i}$) is $B$
(resp. $\D(B)$) for all $i\in{\mathbb{Z}}$. We denote the element of $\widehat{B}$ by $(a_{i},f_{i})_{i\in\mathbb{Z}}, $ where $a_{i}\in B_i$ and $f_{i}\in Q_i$. The multiplication in $\widehat{B}$ is defined to be
\[(a_{i},f_{i})_{i}\cdot(b_{i},g_{i})_{i}=(a_ib_i,a_{i}g_{i}+f_{i}b_{i+1})_{i}, 
\]
for $a_{i},b_{i}\in B_i$, $f_{i},g_{i}\in Q_i$ and $i\in\mathbb{Z}$.  A group $G$ of automorphisms of $\widehat B$ is said to be {\bf admissible}, if $G$ acts freely on the set of objects of $\widehat B$ and has finitely many orbits.  {\bf Nakayama automorphism} $\nu_{\widehat B}$ is defined by the identity shift $B_{i}\rightarrow B_{i+1}$ and $Q_{i}\rightarrow Q_{i+1}$ for any integer $i$. Take a complete set  $\{e_{i}\mid i=1,2,\cdots,n\}$ of primitive orthogonal idempotent elements of $B$. Then there is a set  $\{e_{(i,m)}\mid i=1,2,\cdots,n, m\in\mathbb{Z}\}$ of primitive orthogonal idempotent elements of $\widehat{B}$.  Note that  $\nu _{\widehat{B}}({\widehat B}e_{(i,m)})=\widehat{B}e_{(i,m+1)}$  for any integers $i$ and $m$.  The infinite cyclic group $<\nu_{\widehat{B}}>$ generated by $\nu_{\widehat{B}}$ is an admissible group of automorphism of  $\widehat{B}$. It is well-known that the orbit category $\widehat{B}/<\nu_{\widehat{B}}> $ is the trivial extension $T(B)=B\ltimes D(B)$ and there is a canonical Galois covering 
\[F:\widehat{B}\rightarrow \widehat{B}/<\nu_{\widehat{B}}>\cong T(B).\]

Recalled that an algebra $B$ is an {\bf Euclidean algebra}, if $B$ is a representation-infinite tilted algebra of the Euclidean type $\widetilde{A}_{m}$($m\geq1$), $\widetilde{D}_{n}$ ($n\geq4$), $\widetilde{E}_{6}$, $\widetilde{E}_{7}$, $\widetilde{E}_{8}$, having a complete section in the preinjective component (please refer to \cite{Ri}).  
We assume that $B$ is triangular, that is, the Gabriel quiver  $Q_{B}$ has no oriented cycles. We identity $B$ with the full bounded subcategory of $\widehat{B}$ given by the objects  $e_{(0,i)}$ for $1\leq i\leq n$. For a sink $i$ of $Q_{B}$, the {\bf reflection }$S^{+}_{i}B$ of $B$ at $i$ \cite{HW} is the full subcategory of $\widehat{B}$ given by the objects
\[e_{(0,j)}, 1\leq j\leq n, j\neq i,\ \  \text{and}\ \   e_{(1,i)}=\nu_{\widehat{B}}(e_{(0,i)}). \]
Then $\sigma^{+}_{i}Q_{B}=Q_{S^{+}_{i}B}$ is called the reflection of $Q_{B}$ at $i$. 
An Euclidean algebra $B$ is said to be {\bf exceptional} if there exists a reflection sequence of sinks $i_{1}, i_{2},\cdots,i_{t}$ in $Q_{B}$  such that $t<r(K_{0}(B))$ and $B\cong S^{+}_{i_{t}}\cdots S^{+}_{i_{2}}S^{+}_{i_{1}}B$, where $r(K_{0}(B))$ is the dimension of Grothendieck group of $B$.

\begin{Them}$($\cite{BoS}, \cite[Theorem 4.9]{S}$)$\label{Euclidean-exceptional}
Let $B$ be an Euclidean  algebra. Then the following statements are equivalent.
\begin{enumerate}[$(1)$]
\item $B$ is an exceptional algebra.
\item There is an automorphism of $\varphi$ of $\widehat{B}$ with  $\varphi^{2}=\nu_{\widehat{B}}$.	
\end{enumerate}
\end{Them}

The following theorem provide us with an equivalent definition of  1-domestic Brauer graph algebra. 

\begin{Them}$($\cite[Theorem 1]{BoS}$)$\label{symmetric-algebras-of-Euclidean-type-1}
Let $A$ be a basic indecomposable algebra. Then the following conditions are equivalent.
\begin{enumerate}[$(1)$]
\item $A$ is isomorphic to  a  1-domestic Brauer graph algebra. 	
\item $A$ is a symmetric algebra of Euclidean type $\widetilde{\mathbb{A}}_m$  and the Cartan matrix of $A$ is non-singular.
\item $A$ is isomorphic to an algebra of the form $\widehat B/<\varphi>$, where $B$ is a  representation-infinite tilted algebra of Euclidean type  $\widetilde{\mathbb{A}}_m$  and $\varphi$ is a square root  of Nakayama automorphism $\nu_{\widehat B}$ of  $\widehat B$, but $A$ is not isomorphic to the four-dimensional local algebra $K<X,Y>/<X^2,Y^2,XY+YX>$, if char $k\neq 2$.
\end{enumerate}
\end{Them}

\begin{Them}$($\cite[Theorem 2]{BoS}$)$\label{symmetric-algebras-of-Euclidean-type-2}
Let $A$ be a basic indecomposable algebra. Then the following conditions are equivalent.
\begin{enumerate}[$(1)$]
\item $A$ is isomorphic to  a  2-domestic Brauer graph algebra. 	
\item $A$ is a symmetric algebra of Euclidean type $\widetilde{\mathbb{A}}_m$ and the Cartan matrix of $A$ is singular.
\item $A$ is isomorphic to the trivial extension $T(B)$, where $B$ is a  representation-infinite tilted algebra of Euclidean type $\widetilde{\mathbb{A}}_m$.		
	\end{enumerate}
\end{Them}

The following theorem presents a full characterization for domestic Brauer graph algebras. 
\begin{Them}$($\cite{BoS}, \cite[Theorem 5.1]{S0}$)$\label{domestic-Brauer-graph-algebra}
Let $A$ be a Brauer graph algebra with Brauer graph $G$. Then
\begin{enumerate}[$(1)$]
\item $A$  is $1$-domestic if and only if one of the following holds:
\begin{enumerate}[$(a)$]
\item $G$ is a tree with $m(i)=2$ for exactly two vertices $i_{0}, i_{1}\in G_{0}$ and $m(i)=1$ for all $i\in G_{0}, i\neq i_{0}, i_{1}$.
\item There is exactly one cycle in $G$ and this cycle has odd length, and all multiplicities of edges are one, that is, $m\equiv1$.
\end{enumerate}
\item $A$ is $2$-domestic if and only if  there is exactly one cycle in $G$ and this cycle has even length, and $m\equiv1$.
\item There is no $n$-domestic  Brauer graph algebra for $n\geq3$.
\end{enumerate}
\end{Them}

\begin{Them}$($\cite[Theorem 4.4, Corollary 4.5]{D1}$)$\label{two-domestic-Brauer-graph-algebras} 
Let $A$ be a representation-infinite domestic Brauer graph algebra with Brauer graph $G$ with $n$ edges. If $G$ has a cycle, then it is unique. Let $n_{1}$ be the number of {\rm(}additional{\rm)} edges on the inside of the cycle and $n_{2}$ the number of {\rm(}additional{\rm)} edges on the outside of the cycle.
In the notation above, 
\begin{enumerate}[$(1)$]
\item If $A$ is $1$-domestic, then $m=1$ and $p+q=2n$. Furthermore,
\begin{enumerate}[$(a)$]
	\item If $G$ is a tree, then $p=q=n$. 
	\item If $G$ has exactly one cycle in $G$ and this cycle has {\rm(}odd{\rm)} length $\ell$, then $p=\ell+2n_{1}$ and $q=\ell+2n_{2}$.
\end{enumerate}
\item If $A$ is $2$-domestic and the unique cycle of $Q$ has {\rm(}even{\rm)} length $\ell$, then $m=2$ and $p=\ell/2+n_{1}$ and $q=\ell/2+n_{2}$ such that $p+q=n$.
\end{enumerate}
\end{Them}

\subsection{Covering theory}
We briefly recall the definition of covering functor and Galois covering.
\begin{Def}[{\cite[Definition 3.1]{BG}}] \label{covering-functor}
Let $F:\mathcal{C}\rightarrow \mathcal{D}$ be a k-linear functor between two k-categories. F is called a {\bf covering functor} if the maps
\begin{equation}\bigoplus_{Fz=b} \mathcal{C}(x,z)\rightarrow \mathcal{D}(a,b) \text{ and }  \bigoplus_{Ft=a} \mathcal{C}(t,y)\rightarrow \mathcal{D}(a,b)
\end{equation}
which are induced by $F$, are bijective for any two objects $a$ and $b$ of $\mathcal{D}$. Here $t$ and $z$ range over all objects of $\mathcal{C}$ such that $Fz=b\ and\  Ft=a$ respectively; 
The maps are supposed to be bijective for all $x$ and $y$ chosen among the $t$ and $z$ respectively.
\end{Def}	

\begin{Def} [{\cite[Section 3.1]{G}}]\label{galois-covering-functor}
Let $M,N$ be locally finite dimensional categories. A covering functor $F: M\rightarrow N$ is called a {\bf Galois covering} if $F$ is surjective on objects and there exists a group $G$ of $k$-linear automorphisms of $M$ which acts freely on (objects of) $M$, such that $F\circ g=F$ for any $g\in G$ and $G$ acts transitively on $F^{-1}(a)$ for each object $a$ of $N$.
\end{Def}

By Theorem \ref{Euclidean-exceptional} and Theorem \ref{symmetric-algebras-of-Euclidean-type-1},  one  1-domestic Brauer graph algebra $A$ is of the form $\widehat{B}/G$, where $B$ is an exceptional Euclidean algebra and $G$ is an infinite cyclic group  $<\varphi>$ generated by a square root $\varphi$ of Nakayama automorphism $\nu_{\widehat{B}}$ of $\widehat{B}$. We know that  $G$ is an admissible group of automorphism of  $\widehat{B}$ and there is a canonical Galois covering 
\[F:\widehat{B}\rightarrow A=\widehat{B}/<\varphi>.\]
Since $\widehat{B}$ is locally support-finite, by Dowbor-Skowronski  \cite[Main Theorem]{DS}, the push down functor $F_{\lambda}\colon  \widehat{B}$-$\m$$\rightarrow A$-$\m$ induced by $F$ is dense. Note that the push down functor $F_{\lambda}$ is exact (see \cite[Section II.1]{LLXC} for more details).  Since $F_{\lambda}$ preserves Auslander-Reiten sequences (please refer to \cite[Section 2.7]{G}), we have \[\Gamma_{A}=\Gamma_{\widehat{B}/<\varphi>}=\Gamma_{\widehat{B}}/<\varphi>.\]

It is known \cite{ANS} that  the AR-quiver $\Gamma_{\widehat{B}}$ of  $\widehat{B}$ is of the form
\begin{equation}\label{AR-quiver-1}
\Gamma_{\widehat{B}}=\bigvee_{n\in\mathbb{Z}}(\mathcal{Y}_{n}\vee{\mathcal{C}_{n}})
\end{equation}
where stable connected  component $_{s}\mathcal{Y}_{i}=\mathbb{Z}\widetilde{A}_{m}$  are Euclidean components and $_{s}\mathcal{C}_{i}$ are $\mathbb{P}_{1}(k)$-families of quasi-tubes of the same tubular type. The action of Nakayama automorphism $\nu_{\widehat{B}}$ in the connected component of  $\Gamma_{\widehat{B}}$  states as follows.  
\begin{equation}\label{AR-quiver-2}
\nu_{\widehat{B}}(\mathcal{Y}_{i})={\mathcal{Y}_{i+2}},\ \text{and} \ \nu_{\widehat{B}}(\mathcal{C}_{i})={\mathcal{C}_{i+2}}
\end{equation}for  any  $i\in\mathbb{Z}$.

Since $\varphi$ is a square root of Nakayama automorphism $\nu_{\widehat{B}}$, there is  a self-equivalence of module category $\widehat{B}$-$\m$ of the algebra $\widehat{B}$ induced by  $\varphi$, still denoted by $\varphi$ if there is no confusion. Therefore it induces a stable equivalence $\overline{\varphi}$ of   $\widehat{B}$-$\stmod$:
\[\overline{\varphi}\colon \widehat{B}{\text-}\stmod\longrightarrow\widehat{B}{\text-}\stmod.\]
It follows that $\overline{\varphi}$ commutes with the syzygy functor $\Omega_{\widehat{B}}$, that is,  $\overline{\varphi}\Omega_{\widehat{B}}=\Omega_{\widehat{B}}\overline{\varphi}$. Thus the following diagram is commutative:
\begin{equation}\label{comm-diagram}
\begin{aligned}
\xymatrix{\widehat{B}{\text-}\stmod \ar[r]^-{\overline{\varphi}} \ar[d]_-{\Omega_{\widehat{B}}}&\widehat{B}{\text-}\stmod\ar[d]^-{\Omega_{\widehat{B}}} \\
\widehat{B}{\text-}\stmod  \ar[r]^-{\overline{\varphi}}&\widehat{B}{\text-}\stmod.}
\end{aligned}
\end{equation}
By equation (\ref{AR-quiver-2}),  the action of $\overline{\varphi}$ satisfies 
\begin{equation}\label{AR-quiv-3}
\overline{\varphi}( {_{s}\mathcal{Y}_{i}})={_{s}\mathcal{Y}_{i+1}}\  \text{and} \ \  \overline{\varphi}({_{s}\mathcal{C}_{i}})={_{s}\mathcal{C}_{i+1}}
\end{equation}
for any $i\in\mathbb{Z}.$ Note that $\overline{\varphi}$ sends $\tau$-periodic (resp. non-$\tau$-periodic) modules to $\tau$-periodic (resp. non-$\tau$-periodic) modules.
Moreover, $F_{\lambda}$ induces a covering functor  from $\widehat{B}$-$\stmod$ to $A$-$\stmod$.
Thus there are isomorphism as follows.
\begin{equation}\label{covering-iso-1}
	\begin{aligned}
		&\StHom_{A}(M,N)\cong\bigoplus_{g(X)=M,g\in G}\StHom_{\widehat{B}}(X,N)=\bigoplus_{i\in \mathbb{Z}}\StHom_{\widehat{B}}(\overline{\varphi}^{i}(M),N),\\
		&\StHom_{A}(M,N)\cong\bigoplus_{g(Y)=N,g\in G}\StHom_{\widehat{B}}(M,Y)=\bigoplus_{i\in \mathbb{Z}}\StHom_{\widehat{B}}(M,\overline{\varphi}^{i}(N)).
	\end{aligned}
\end{equation} 

\section{Simple-minded systems over 1-domestic Brauer algebras}
In this section, we  study simple-minded systems over  1-domestic Brauer algebras by covering theory and the characterization of simple-minded systems over  2-domestic Brauer algebras in \cite{Z1}.
\begin{Lem}\label{covering-1-2-domestic}
Let $A$ be a $1$-domestic Brauer graph algebra.  Then there is a $2$-domestic Brauer graph algebra $C$ and a dense covering functor $\overline{F}: C$-$\stmod\rightarrow A$-$\stmod$.
\end{Lem}
\begin{proof}
By Theorem  \ref{Euclidean-exceptional} and \ref{symmetric-algebras-of-Euclidean-type-1},  there is  an exceptional Euclidean algebra $B$ of $\widetilde{A}_{m}$ such that $A=\widehat{B}/<\varphi>$,  where $\varphi$ is a square root of the Nakayama automorphism $\nu_{\widehat{B}}$ of repetitive algebra $\widehat{B}$.  Since  infinite cyclic groups $<\varphi>$ and $<\nu_{\widehat{B}}>$ are both admissible, there are two canonical Galois coverings  $F_{1}: \widehat{B}\rightarrow \widehat{B}/<\nu_{\widehat{B}}>$  and  $F_{2}: \widehat{B}\rightarrow A$.  Since $\widehat{B}$ is locally support-finite, by Dowbor-Skowronski  \cite[Main Theorem]{DS},  the push-down functors 
from $\widehat{B}\text{-}\m$ to $\widehat{B}/<\nu_{\widehat{B}}>\text{-}\m$ and  from $\widehat{B}\text{-}\m$ to $A\text{-}\m$ 
induced by $F_{1}$ and $F_{2}$, respectively, are dense, still denoted by $F_{1}$ and $F_{2}$ respectively. By Theorem \ref{symmetric-algebras-of-Euclidean-type-2},  $\widehat{B}/<\nu_{\widehat{B}}>$  is a $2$-domestic Brauer graph algebra, denoted by $C$. Since $\varphi$ is a square root of $\nu_{\widehat{B}}$, there is a Galois covering  $F: C\rightarrow C/<\varphi>$ induced by $\varphi$. It is clear  that $A\cong C/<\varphi>$. Thus there is a canonical Galois covering  from $C$  to $A$, still denoted by $F$, see the following diagram. \[\xymatrix{ 
	\widehat{B} \ar[d]^-{F_{1}}\ar[drr]^-{F_{2}} \\
	C \ar@{-->}[rr]_-{F} &&  A. 
}\]
The push down functor from $C$-$\m$  to $A$-$\m$ induced by $F$ is a dense covering functor, still denoted by $F$. Since $F$ preserves projective modules(please refer to \cite[Proposition 3.2]{BG}), there is a dense covering functor $\overline{F} $ from $C$-$\stmod$ to $A$-$\stmod$ induced by $F$.
\end{proof}

Let $B$ be an exceptional Euclidean algebra and $S$ a family of objects in $\widehat{B}$-$\stmod$.  By Theorem \ref{Euclidean-exceptional}, there is a square root $\varphi$ of Nakayama automorphism $\nu_{\widehat{B}}$ of $\widehat{B}$. $S$ is said to be {\bf $\overline{\varphi}$-stable}  in $\widehat{B}$-$\stmod$, if $\overline{\varphi}(S)\subseteq S$. 
We consider the induced functor of $\overline{\varphi}$ in $\widehat{B}/<\nu_{\widehat{B}}>$-$\stmod$, still denoted by $\overline{\varphi}$. For a   family of objects $\mathcal{S}$ in $\widehat{B}/<\nu_{\widehat{B}}>$-$\stmod$, we  say $\mathcal{S}$ is  {\bf $\overline{\varphi}$-stable}  in $\widehat{B}/<\nu_{\widehat{B}}>$-$\stmod$, if $\overline{\varphi}(\mathcal{S})\subseteq \mathcal{S}$. 
 
\begin{Them}\label{sms-1-2-domestic}
Let $A$ be a $1$-domestic Brauer graph algebra and  $\overline{F}$ the covering functor from $C$-$\stmod$ to $A$-$\stmod$  in Lemma \ref{covering-1-2-domestic}. 
\begin{enumerate}[$(1)$]
\item If  $\mathcal{S}$ be a $\overline{\varphi}$-stable simple-minded system in $C$-$\stmod$, then $\overline{F}(\mathcal{S})$ is a simple-minded system in $A$-$\stmod$.
\item If  $\mathcal{M}$ is a simple-minded system in $A$-$\stmod$, then $\overline{F}^{-1}(\mathcal{M})$ is a $\overline{\varphi}$-stable simple-minded system in $C$-$\stmod$.
\end{enumerate}
\end{Them}
\begin{proof}
By Theorem  \ref{Euclidean-exceptional} and \ref{symmetric-algebras-of-Euclidean-type-1}, $A$ is isomorphic to  $\widehat{B}/<\varphi>$, where $B$ is an exceptional Euclidean algebra of type $\widetilde{A}_{m}$ and $\varphi$ is a square root of the Nakayama automorphism $\nu_{\widehat{B}}$ of  $\widehat{B}.$
Let $C$ be the trivial extension $T(B)$ of $B$. By Lemma \ref{covering-1-2-domestic}, there is a dense covering functor $\overline{F}$ from $C$-$\stmod$ to $A$-$\stmod$.
It is known that  the stable AR-quiver $_{s}\Gamma_{A}$ contains a stable Euclidean component of the form $\mathbb{Z}A_{p,q}$ 
and the stable AR-quiver $_{s}\Gamma_{C}$ contains two stable Euclidean components of the form $\mathbb{Z}A_{p,q}$. 
By Theorem \ref{two-domestic-Brauer-graph-algebras},  $n=(p+q)/2$ is the number of non-isomorphic simple $A$-modules and $2n=p+q$ is the number of non-isomorphic simple $C$-modules.

(1) Let $\mathcal{S}$ be a $\overline{\varphi}$-stable simple-minded system in $C$-$\stmod$. By Theorem \ref{sms-2Domes-BGA}, $\mathcal{S}$ contains at least one object for each Euclidean component and the cardinality of $\mathcal{S}$ is $2n$, the number of non-isomorphic simple $C$-modules. Without loss of generality, we assume that  $\mathcal{S}=\{S_{1},S_{2},\cdots, S_{2n}\}$.
Since $\mathcal{S}$  is $\overline{\varphi}$-stable, $\overline{\varphi}(\mathcal{S})\subseteq\mathcal{S}.$
 By the action of $\overline{F}$, $\overline{F}(S_{i})=\overline{F}(\overline{\varphi}(S_{i}))$ for each $i$.  By equation (\ref{AR-quiv-3}),  $\overline{\varphi}$ sends a non-periodic module to a non-periodic module. Therefore 
$\overline{F}(\mathcal{S})$ consists of  $n$ objects in $A$-$\stmod$ and 
contains at least an object for the Euclidean component of the stable  AR-quiver  $_{s}\Gamma_{A}$.
By covering theory,
there are  isomorphism as follows.
\begin{equation}
\begin{aligned}\label{covering}
\StHom_{A}(\overline{F}(S_{i}),\overline{F}(S_{j}))&\cong\bigoplus_{\ell\in\mathbb{Z}}\StHom_{\widehat{B}}(\overline{\varphi}^{\ell}(S_{i}),S_{j})\\
&=(\bigoplus_{\ell\in\mathbb{Z}}\StHom_{\widehat{B}}(\nu_{\widehat{B}}^{\ell}(S_{i}),S_{j}))\bigoplus(\bigoplus_{\ell\in\mathbb{Z}}\StHom_{\widehat{B}}(\overline{\varphi}\nu_{\widehat{B}}^{\ell}(S_{i}),S_{j})),
\end{aligned}
\end{equation}

\begin{equation}
\begin{aligned}\label{covering-0}
\StHom_{C}(S_{i},S_{j})&\cong\bigoplus_{\ell\in\mathbb{Z}}\StHom_{\widehat{B}}(\nu_{\widehat{B}}^{\ell}(S_{i}),S_{j})
\end{aligned}
\end{equation}
and 
\begin{equation}
\begin{aligned}\label{covering-1}
\StHom_{C}(\overline{\varphi}^{-1}(S_{i}),S_{j})&\cong\bigoplus_{\ell\in\mathbb{Z}}\StHom_{\widehat{B}}(\nu_{\widehat{B}}^{\ell}(\overline{\varphi}^{-1}S_{i}),S_{j})\\
&=\bigoplus_{\ell\in\mathbb{Z}}\StHom_{\widehat{B}}(\overline{\varphi}\nu_{\widehat{B}}^{\ell}(S_{i}),S_{j}).
\end{aligned}
\end{equation}
 
If  $\overline{F}(S_{i})\ncong\overline{F}(S_{j})$, then $i\neq j$, $\overline{\varphi}^{-1}(S_{i})\neq S_{j}$ and  $\overline{\varphi}(S_{i})\neq S_{j}$ in $C$-$\stmod$.  		
Since $\mathcal{S}$ is $\overline{\varphi}$-stable, both $\overline{\varphi}(S_{i})$ and $\overline{\varphi}^{-1}(S_{i})$ are  in  $\mathcal{S}$. It follows   from the orthogonality of  $\mathcal{S}$ and equations (\ref{covering-0}) and (\ref{covering-1}) that  \[\StHom_{A}(\overline{F}(S_{i}),\overline{F}(S_{j}))\cong(\bigoplus_{\ell\in\mathbb{Z}}\StHom_{\widehat{B}}(\nu_{\widehat{B}}^{\ell}(S_{i}),S_{j}))\bigoplus(\bigoplus_{\ell\in\mathbb{Z}}\StHom_{\widehat{B}}(\overline{\varphi}\nu_{\widehat{B}}^{\ell}(S_{i}),S_{j}))=0.\]
 
 If  $\overline{F}(S_{i})\cong\overline{F}(S_{j})$, then $i=j$ or $\overline{\varphi}(S_{i})=S_{j}$ or $\overline{\varphi}(S_{j})=S_{i}$.  

{\bf Case one:} $i=j$. Then $\overline{\varphi}(S_{i})\neq S_{i}$. Thus 
\begin{equation}
\begin{aligned}\label{covering-2}
\StHom_{A}(\overline{F}(S_{i}),\overline{F}(S_{i}))\cong&(\bigoplus_{\ell\in\mathbb{Z}}\StHom_{\widehat{B}}(\nu_{\widehat{B}}^{\ell}(S_{i}),S_{i}))\bigoplus(\bigoplus_{\ell\in\mathbb{Z}}\StHom_{\widehat{B}}(\overline{\varphi}\nu_{\widehat{B}}^{\ell}(S_{i}),S_{i}))\\
=&\StHom_{C}(S_{i},S_{i})\bigoplus\StHom_{C}(\overline{\varphi}^{-1}(S_{i}),S_{i})\\
=&\StHom_{C}(S_{i},S_{i})\cong k.
\end{aligned}
\end{equation}

{\bf Case two:} $\overline{\varphi}^{-1}(S_{i})=S_{j}$. Then  $i\neq j$. Thus 
\begin{equation}
\begin{aligned}\label{covering-3}
\StHom_{A}(\overline{F}(S_{i}),\overline{F}(S_{j})\cong&(\bigoplus_{\ell\in\mathbb{Z}}\StHom_{\widehat{B}}(\nu_{\widehat{B}}^{\ell}(S_{i}),S_{j}))\bigoplus(\bigoplus_{\ell\in\mathbb{Z}}\StHom_{\widehat{B}}(\overline{\varphi}\nu_{\widehat{B}}^{\ell}(S_{i}),S_{j}))\\
=&\StHom_{C}(S_{i},S_{j})\bigoplus\StHom_{C}(\overline{\varphi}^{-1}(S_{i}),S_{j})\\
\cong&\StHom_{C}(S_{j},S_{j})\cong k.
\end{aligned}
\end{equation}
The case $\overline{\varphi}(S_{j})=S_{i}$ is similar to the Case two.  Thus $\overline{F}(\mathcal{S})$ is an orthogonal system in $A$-$\stmod$.
Since the covering functor $F$ is exact and the extension closure  $\mathcal{F}(\mathcal{S})=C$-$\stmod$, it is routine to check that  the extension closure of $\mathcal{F}(\overline{F}(\mathcal{S}))$ is the stable module category $A$-$\stmod$.  Note that push down functor is an exact functor, please refer to \cite[Section 3.2]{BG} or \cite[Chapter II, Section 2.1]{LLXC} for more details. Thus $\overline{F}(\mathcal{S})$ is a simple-minded system in $A$-$\stmod$.

\noindent(2) Let $\mathcal{M}$ be a simple-minded system in $A$-$\stmod$. Without loss of generality, we assume that $\mathcal{M}=\{M_{1},M_{2}, \cdots,M_{s}\}$.   Then we may assume that \[\overline{F}^{-1}(\mathcal{M})=\{U_{1},V_{1},U_{2},V_{2}, \cdots, U_{s},V_{s}\}\] satisfying  $\overline{F}(U_{i})=\overline{F}(V_{i})=M_{i}$ for $i=1,2,\cdots,s$.  Thus $\overline{\varphi}(U_{i})=V_{i}$ or $\overline{\varphi}(V_{i})=U_{i}$ for  $i=1,2,\cdots,s$. Without loss of generality, we may  assume that $\overline{\varphi}(U_{i})=V_{i}$  for  $i=1,2,\cdots,s$. It follows that $\overline{F}^{-1}(\mathcal{M})$ is $\overline{\varphi}$-stable.  

By covering theory, 
\begin{equation}
\begin{aligned}\label{covering-4}
\StHom_{A}(M_{i},M_{j})\cong&\StHom_{C}(U_{i},U_{j})\bigoplus\StHom_{C}(V_{i},U_{j})\\
\cong&\StHom_{C}(V_{i},U_{j})\bigoplus\StHom_{C}(V_{i},V_{j}).
\end{aligned}
\end{equation}
Since $\mathcal{M}$ is an orthogonal system in $A$-$\stmod$, 
 $\StHom_{C}(U_{i},U_{i})\cong\StHom_{C}(V_{i},V_{i})\cong k$ for $i=1,2,\cdots,s$, and $\StHom_{C}(U_{i},V_{j})\cong\StHom_{C}(V_{j},U_{i})\cong 0$ for  $i\neq j$. Thus $\overline{F}^{-1}(\mathcal{M})$ is an orthogonal system in $C$-$\stmod$. Since $F$ is exact and dense, the extension closure of $\overline{F}^{-1}(\mathcal{M})$ is the stable module category  $C$-$\stmod$. Thus $\overline{F}^{-1}(\mathcal{M})$ is a $\overline{\varphi}$-stable simple-minded system in $C$-$\stmod$.
\end{proof}

\begin{Cor}\label{sms-BGA}
Let $A$ be a 1-domestic Brauer graph algebra and $\mathcal{S}$ a maximal orthogonal system which contains at least one object for the Euclidean component. Then $\mathcal{S}$ is a simple-minded system in $A$-$\stmod$ and the cardinality of $\mathcal{S}$ is the number of non-projective non-isomorphic simple $A$-modules.
\end{Cor}
\begin{proof}
Let $\overline{F}$ be the covering functor from $C$-$\stmod$ to $A$-$\stmod$ in Lemma \ref{covering-1-2-domestic}.  By  Theorem  \ref{sms-1-2-domestic}, $\overline{F}^{-1}(\mathcal{S})$  is a maximal orthogonal system in $C$-$\stmod$ which contains at least one object for each Euclidean component. By Theorem \ref{sms-2Domes-BGA}, $\overline{F}^{-1}(\mathcal{S})$ is a simple-minded system in $C$-$\stmod$ and the cardinality of $\overline{F}^{-1}(\mathcal{S})$ is $2n$. By (1) of Theorem \ref{sms-1-2-domestic}, $\mathcal{S}$ is  a simple-minded system in $A$-$\stmod$ and the cardinality of $\overline{F}^{-1}(\mathcal{S})$ is $n$, the number of  non-projective  non-isomorphic simple $A$-modules. 
\end{proof}

Note that one domestic Brauer graph algebra is 1-domestic or 2 domestic. By  Corollary \ref{sms-BGA} and Theorem \ref{sms-2Domes-BGA}, the following conclusion holds. 
\begin{Cor}\label{sms-BGA-1}
Let $A$ be a domestic Brauer graph algebra and $\mathcal{S}$  a family of objects in $A$-$\stmod$.
Then $\mathcal{S}$ is a simple-minded system in $A$-$\stmod$ if and only if $\mathcal{S}$ is a maximal orthogonal system which contains at least one object for each Euclidean component.
\end{Cor}

\begin{Cor}
Let $A$ be a domestic Brauer graph algebra  and $\mathcal{S}$ a weakly simple-minded system with a finite cardinality. Then $\mathcal{S}$ is a simple-minded system in $A$-$\stmod$.
\end{Cor}

\section{The action of $\overline{\varphi}$ in the stable AR-quiver $_{s}\Gamma_{\widehat{B}}$}
Let $A$ be a 1-domestic Brauer graph algebra. Stated as in 	Section 2, $A$ is of the form  $\widehat{B}/G$, where $B$ be an exceptional Euclidean algebra of type $\widetilde{A}_{m}$ and $G$ is an infinite cyclic group generated by a square root $\varphi$ of Nakayama automorphism $\nu_{\widehat{B}}$.
By equation (\ref{AR-quiver-1}), the stable AR-quiver $_{s}\Gamma_{\widehat{B}}$ is of the form  
\begin{equation}\label{repet-AR-qui}
\begin{aligned}
\bigvee_{i\in\mathbb{Z}}(_{s}\Gamma_{i}\vee{_{s}\mathcal{C}_{i}})=\bigvee_{i\in\mathbb{Z}}(_{s}\Gamma_{i}\vee {_{s}P_{i}}\vee {_{s}Q_{i}}\vee{_{s}\mathcal{H}_{i}}),
\end{aligned}
\end{equation}
where $_{s}\Gamma_{i}$ is the stable Euclidean component of the form $\mathbb{Z}\widetilde{A}_{p,q}$; the family $_{s}\mathcal{C}_{i}= {_{s}P_{i}}\vee {_{s}Q_{i}}\vee{_{s}\mathcal{H}_{i}}$  is stable generalized standard for each $i$, where $P_{i}$ is a quasi-tube of rank $p$,  $Q_{i}$ is a quasi-tube of rank $q$ and  $\mathcal{H}_{i}$ is a family of homogeneous tubes.  In this section,  we shall determine the action of $\overline{\varphi}$ in the stable AR-quiver $_{s}\Gamma_{\widehat{B}}$.  

Let $B$ be an Euclidean algebra. Equivalently, $B$ is a tilted algebra of Euclidean type, thus $\mathcal{D}^{b}(B)$ is triangulated equivalent to the derived category $\mathcal{D}^{b}(H)$ of a hereditary algebra $H$ of Euclidean type. Since the global dimension of $B$ is finite (less than or equal to two), $\mathcal{D}^{b}(B)$  and  $\widehat{B}$-$\stmod$ are triangulated equivalent. Thus $\widehat{B}$-$\stmod$ is triangulated equivalent to the derived category $\mathcal{D}^{b}(H)$. Therefore it is not hard to know that every connected component of the stable AR-quiver of $\widehat{B}$ is stable generalized standard.  

\begin{Prop}$($\cite[Proposition 2.14]{Z1}$)$\label{repetitive-stable generalized}
Let $B$ be an Euclidean algebra and $\widehat{B}$ the repetitive algebra of $B$. Then every connected component of the AR-quiver $\Gamma_{\widehat{B}}$ of $\widehat{B}$ is stable generalized standard. 
\end{Prop}

We use the same notations as in \cite[Section 1]{Z1} to label objects in different components of $_{s}\Gamma_{\widehat{B}}$. 
Recall  that, for every Euclidean component $_{s}\Gamma_{i}$ of  $_{s}\Gamma_{\widehat{B}}$, we use  notation $(i,j,k)$ to denote a non-projective object in $_{s}\Gamma_{i}$, where $i$ distinguishes different Euclidean components of $_{s}\Gamma_{\widehat{B}}$, $j$ and $k$ are  determined by $\mathbb{Z}\times\Delta_{0}.$ Note that the AR-translation $\tau$  in  $_{s}\Gamma_{i}$ satisfies $\tau(i,j,k)=(i,j-1,k-1).$ 
We  assume that  \[\Omega^{-1}_{\widehat{B}}(i,j,k)=(i+1,j,k)\ \text{and}\ \Omega_{\widehat{B}}(i,j,k)=(i-1,j,k).\] 
Without loss of generality, we always assume that $\Omega_{\widehat{B}}^{-1}$ preserves the orientation of irreducible maps in $_{s}\Gamma_{i}$ for each $i$, that is, if there is an irreducible map $\alpha\colon(i,j,k)\rightarrow(i,j+1,k)$ (or $\beta\colon (i,j,k)\rightarrow(i,j,k+1)),$  then
$\Omega_{\widehat{B}}^{-1}(\alpha)$ is an irreducible map from  $(i+1,j,k)$ to $(i+1,j+1,k)$ (or $\Omega_{\widehat{B}}^{-1}(\beta)$ is an irreducible map from $(i+1,j,k)$ to $(i+1,j,k+1)$).

For a quasi-tube $P_{i}$ (resp. $Q_{i}$) of  $_{s}\Gamma_{\widehat{B}}$, we use the notation $P(i,j,k)$ (resp. $Q(i,j,k))$ to denote a non-projective object on $P_{i}$ (resp. $Q_{i}$), where $i$ distinguishes different quasi-tubes of $_{s}\Gamma_{\widehat{B}}$, $j$ and $k$ are  determined by $\mathbb{Z} A_{\infty}/<\tau^{r}>$ for a positive integer $r=p$ or $q$.
Note that the AR-translation $\tau$ in the  component  $P_{i}$ or $Q_{i}$ satisfies \[\tau P(i,j,k)=P(i,j-1,k)\ \text{or}\ \tau Q(i,j,k)=Q(i,j-1,k).\] We may still use a pair $(i,j)$ of integers to label an object in a connected component of the stable AR-quiver if there is no confusion.

\begin{Prop}\label{action-of-varphi}
Let $B$ be an exceptional  Euclidean algebra of type $\widetilde{A}_{m}$,  $\varphi$ a square root of Nakayama automorphism $\nu_{\widehat{B}}$ and let $\mathcal{C}$ be an Euclidean component of the form $\mathbb{Z}\widetilde{A}_{p,q}$ of  the AR-quiver $\Gamma_{\widehat{B}}$.  Take a  non-projective object $X=(0,a,b)$ in $\mathcal{C}$. Then 
\begin{enumerate}[$(1)$]
\item 	$\overline{\varphi}(X)=(1,i,j)$, where $a\leq i<a+p$ and $b-q< j<b$. Moreover, if $\overline{\varphi}$ preserves the orientation of irreducible maps in $\mathcal{C}$, then $p$ and $q$ are odd numbers and $i=a+\dfrac{p-1}{2}$ and $j=b-\dfrac{q+1}{2}$. 	
\item 	$\overline{\varphi}^{-1}(X)=(-1,r,s)$, where $a<r<a+p$ and $b-q<s\leq b$.
Moreover, if  $\overline{\varphi}^{-1}$ preserves the orientation of irreducible maps in $\mathcal{C}$, then $p$ and $q$ are odd numbers and $r=a+\dfrac{p+1}{2}$ and $s=b-\dfrac{q-1}{2}$. 	
\end{enumerate}
\end{Prop}
\begin{proof}
We prove only conclusion (1) since $(2)$ is dual to (1).

\noindent (1) Take a non-projective object $X=(0,a,b)$. We assume that $\overline{\varphi}(X)=(1,i,j)$, where $a\leq i<a+p$ and $b<j< b-q$. 
By the action of $\Omega_{\widehat{B}}^{-1}$ and $\nu_{\widehat{B}}$ in the  Euclidean component of $\Gamma_{\widehat{B}}$, we know that \[\Omega_{\widehat{B}}^{-2}(X)=(2,a,b)\ \ \text{and}\ \  \nu_{\widehat{B}}(X)=\overline{\varphi}^{2}(0,a,b)=\overline{\varphi}(1,i,j).\] Since $\nu_{\widehat{B}}(X)=\tau\Omega_{\widehat{B}}^{-2}(0,a,b)=\tau(2,a,b)=(2,a-1,b-1)$, we have 
\begin{equation}\label{equation-1}
\overline{\varphi}(1,i,j)=(2,a-1,b-1).
\end{equation}
It is clear that $\Omega_{\widehat{B}}^{-1}\overline{\varphi}(X)=(2,i,j)$ and $\overline{\varphi}\Omega_{\widehat{B}}^{-1}(X)=\overline{\varphi}(1,a,b).$
Since $\Omega_{\widehat{B}}^{-1}\overline{\varphi}=\overline{\varphi}\Omega_{\widehat{B}}^{-1}$, 
we have 
\begin{equation}\label{equation-2}
\overline{\varphi}(1,a,b)=(2,i,j).
\end{equation}

There are two cases to be considered. 

\noindent {\bf Case one}: $\overline{\varphi}$  preserves the orientation of irreducible maps in $\mathcal{C}$. 

Since both $\overline{\varphi}$ and $\Omega_{\widehat{B}}^{-1}$ preserve the orientation of irreducible maps in $\mathcal{C}$,  
the relative position from object $(1,i,j)$ to object $(1,a,b)$ in the  connected component of stable AR-quiver is the same with from object $\overline{\varphi}(1,i,j)$ to object $\overline{\varphi}(1,a,b)$. 
See the following diagram.
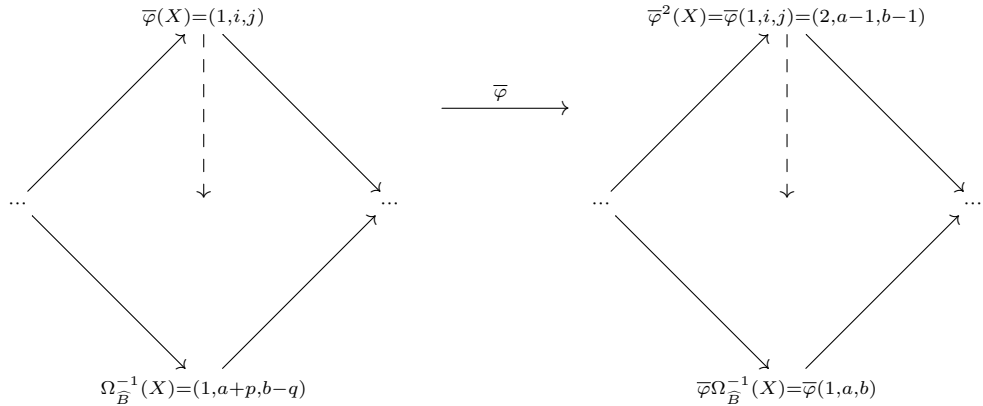
\[\xymatrix@R=35pt@C=35pt@!0{
& &\scriptstyle \overline{\varphi}(X)=(1,i,j)\ar@{->}[ddrr] \ar@{-->}[dd] \\
& & & &\ \ \ \ \ \ \ \ \ \ \ar@{->}[rr]^{\ \ \ \ \overline{\varphi}}&& \\
\scriptstyle\cdots \ar@{->}[ddrr] \ar@{->}[uurr]&&&&\scriptstyle\cdots\\
 &  &    \\
 & &\scriptstyle\Omega_{\widehat{B}}^{-1}(X)=(1,a+p,b-q) \ar@{->}[uurr] &   }
\xymatrix@R=35pt@C=35pt@!0{
& &\scriptstyle \overline{\varphi}^{2}(X)=\overline{\varphi}(1,i,j)=(2,a-1,b-1)\ar@{->}[ddrr]\ar@{-->}[dd] & \\
& & \\
\scriptstyle\cdots \ar@{->}[ddrr] \ar@{->}[uurr] &&&&\scriptstyle\cdots\\
& &  \\
 &&\scriptstyle\overline{\varphi}\Omega_{\widehat{B}}^{-1}(X)=\overline{\varphi}(1,a,b) \ar@{->}[uurr] & & }
\]
It follows  that  $\overline{\varphi}(1,a,b)=(2,2a-1-i+p,2b-1-j-q)$.
Therefore, by equations (\ref{equation-1}) and  (\ref{equation-2}),
\[(2,i,j)=(2,2a-1-i+p,2b-1-j-q).\]

Thus \[2i=2a-1+p \ \text{and} \  2j=2b-1-q.\]
Since both $i$ and $j$ are integers, $p$ and $q$ are odd numbers and $i=a+\dfrac{p-1}{2}$ and $j=b-\dfrac{q+1}{2}$. Note that the action of $\overline{\varphi}$ in the component is totally determined by  integers $p$ and $q$.

\noindent {\bf Case two}: $\overline{\varphi}$  inverses  the orientation of irreducible maps in $\mathcal{C}$.
 
By equation (\ref{equation-1}), 
\[\overline{\varphi}\overline{\varphi}(X)=\nu_{\widehat{B}}(0,a,b)=(2,a-1,b-1).\]
Apply the action of $\overline{\varphi}$ to objects $(1,a,b)$ and  $(1,i,j)$ and see the following diagram.
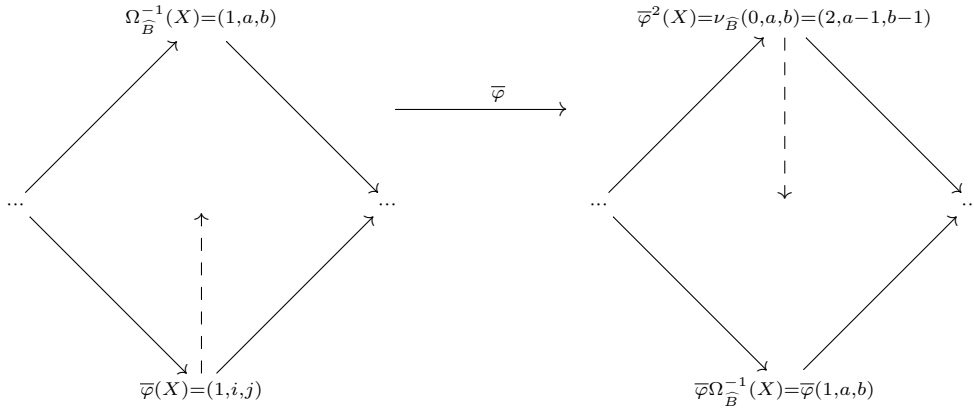
\[\xymatrix@R=35pt@C=35pt@!0{
& & \scriptstyle\Omega_{\widehat{B}}^{-1}(X)=(1,a,b)\ar@{->}[ddrr]  & & \ \ \ \ \ \ \ \ &&  \\
&&&&\ar@{->}[rr]^{\ \ \ \ \overline{\varphi}}&&\\
\scriptstyle\cdots \ar@{->}[ddrr] \ar@{->}[uurr] &  &   &&\scriptstyle\cdots \\
& & & && \\
&&\scriptstyle\overline{\varphi}(X)=(1,i,j) \ar@{->}[uurr]\ar@{-->}[uu] &&& }
\xymatrix@R=35pt@C=35pt@!0{
& & \scriptstyle\overline{\varphi}^{2}(X)=\nu_{\widehat{B}}(0,a,b)=(2,a-1,b-1)\ar@{->}[ddrr] \ar@{-->}[dd]  & & \ \ \ \ \ \ \ \ &&  \\
&&&&\\
\scriptstyle\cdots \ar@{->}[ddrr] \ar@{->}[uurr] &  &   &&\scriptstyle\cdots \\
& &  && \\
&&\scriptstyle\overline{\varphi}\Omega_{\widehat{B}}^{-1}(X)=\overline{\varphi}(1,a,b)  \ar@{->}[uurr] &&& }
\]
Since $\overline{\varphi}$ inverses the irreducible maps in $\mathcal{C}$,  \[\overline{\varphi}\Omega_{\widehat{B}}^{-1}(X)=\overline{\varphi}(1,a,b)=(2,a-1+b-j,b-1+a-i).\] 
By equations (\ref{equation-1})  and (\ref{equation-2}), 
\[(2,i,j)=(2,a-1+b-j,b-1+a-i).\]
Thus $j=b-a+1-i$. Note  $\overline{\varphi}(0,a,b)$ is not a fixed position in this case. 

If $i<a$ and $j<b$ or $i\geq a$ and $j\geq b$, then it is routine to check  that  $\overline{\varphi}\Omega_{\widehat{B}}\neq\Omega_{\widehat{B}}\overline{\varphi}$ for the above two cases. It is a contradiction. 
Thus conclusion (1) holds. 
\end{proof}

The following diagram indicates the position  of object $\overline{\varphi}(0,a,b)$ in the  Euclidean component $_{s}\Gamma{i}$.

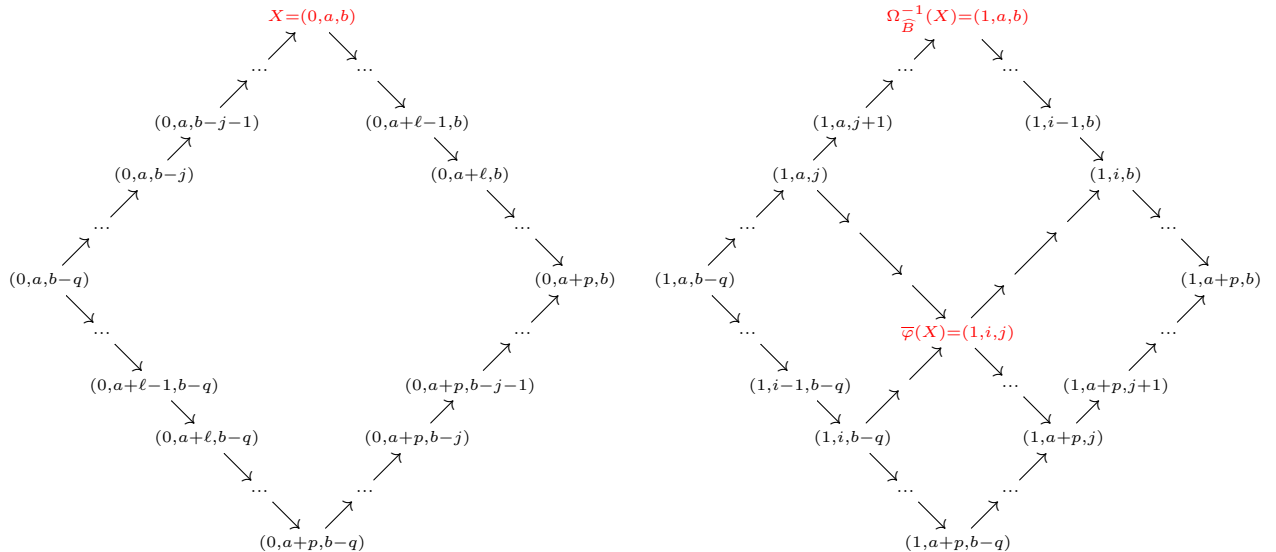
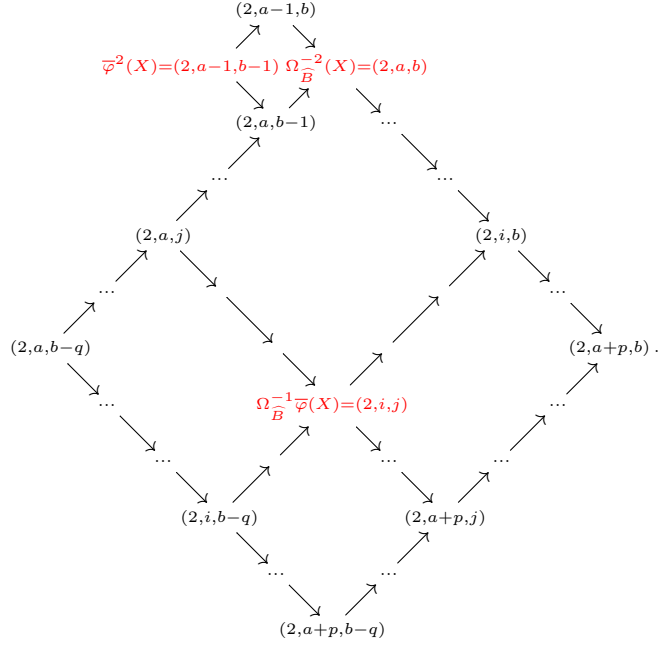
\begin{small}\[\xymatrix@dr@R=14pt@C=14pt@!0{
&&\scriptstyle \old{X=(0,a,b)}\ar[rr] &&\scriptstyle\cdots \ar[rr]& &\scriptstyle (0,a+\ell-1,b)\scriptstyle \ar[rr] && \scriptstyle (0,a+\ell,b)\ar[rr]&& \scriptstyle\cdots\ar[rr]&& \scriptstyle(0,a+p,b)&&\\	
&& \\	
&&\scriptstyle\cdots\ar[uu] &&\scriptstyle&& 	\scriptstyle &&\scriptstyle&& &&\scriptstyle\cdots\ar[uu]&&\\	
&&& \\	
&&	\scriptstyle (0,a,b-j-1)\ar[uu] &&& & \scriptstyle &&\scriptstyle && \scriptstyle &&\scriptstyle(0,a+p,b-j-1)\ar[uu]\\	
&&	& \\
&&\scriptstyle \scriptstyle(0,a,b-j)\ar[uu]&& \scriptstyle &&  && &&\scriptstyle&&\scriptstyle(0,a+p,b-j)\ar[uu]\\
&&& \\
&&\scriptstyle\cdots \ar[uu] && \scriptstyle && \scriptstyle  &&\scriptstyle&&\scriptstyle &&\scriptstyle\cdots\ar[uu]\\ 
&&\\
&&\scriptstyle(0,a,b-q) \ar[uu] \ar[rr]&&\scriptstyle\cdots\ar[rr]&&\scriptstyle(0,a+\ell-1,b-q)\ar[rr]&&\scriptstyle(0,a+\ell,b-q)\ar[rr] &&\scriptstyle\cdots\ar[rr]&&\scriptstyle(0,a+p,b-q) \ar[uu]
}
\xymatrix@dr@R=14pt@C=14pt@!0{
&&\scriptstyle\old{\Omega^{-1}_{\widehat{B}}(X)=(1,a,b)}\ar[rr] &&\scriptstyle\cdots \ar[rr]& &\scriptstyle (1,i-1,b)\scriptstyle \ar[rr] && \scriptstyle(1,i,b)\ar[rr]&& \scriptstyle\cdots\ar[rr]&& \scriptstyle(1,a+p,b)&&\\	
&& \\	
&&\scriptstyle\cdots\ar[uu] &&\scriptstyle&& 	\scriptstyle &&\scriptstyle\ar[uu]&& &&\scriptstyle\cdots\ar[uu]&&\\	
&&& \\	
&&	\scriptstyle (1,a,j+1)\ar[uu] && && &&\scriptstyle \scriptstyle\ar[uu] && \scriptstyle &&\scriptstyle(1,a+p,j+1)\ar[uu]\\	
&&	& \\
&&\scriptstyle(1,a,j)\ar[rr] \ar[uu]&& \scriptstyle\ar[rr] &&\scriptstyle\ar[rr]  && \scriptstyle\old{\overline{\varphi}(X)=(1,i,j)}\ar[rr]\ar[uu] &&\scriptstyle\cdots\ar[rr]&&\scriptstyle(1,a+p,j)\ar[uu]\\
&&& \\
&&\scriptstyle\cdots \ar[uu] && \scriptstyle && \scriptstyle&&\scriptstyle \ar[uu] &&\scriptstyle &&\scriptstyle\cdots\ar[uu]\\ 
&&\\
&&\scriptstyle(1,a,b-q) \ar[uu] \ar[rr]&&\scriptstyle\cdots\ar[rr]&&\scriptstyle(1,i-1,b-q)\ar[rr]&&\scriptstyle(1,i,b-q)\ar[uu]\ar[rr] &&\scriptstyle\cdots\ar[rr]&&\scriptstyle(1,a+p,b-q) \ar[uu]
		}\]
\end{small}
	
\begin{small}
\[\xymatrix@dr@R=15pt@C=15pt@!0{
\scriptstyle(2,a-1,b)\ar[rr]	&&\scriptstyle\ \ \ \ \ \ \old{\Omega^{-2}_{\widehat{B}}(X)=(2,a,b)}\ar[rr] &&\scriptstyle\cdots \ar[rr]& &\scriptstyle \cdots\ar[rr] && \scriptstyle \scriptstyle(2,i,b)\ar[rr]&& \scriptstyle\cdots\ar[rr]&& \scriptstyle(2,a+p,b)\,.&&\\	
&& \\	
\scriptstyle\old{\overline{\varphi}^{2}(X)=(2,a-1,b-1)}\ \ \ \ \ \ \ \ar[uu]\ar[rr]	&&\scriptstyle(2,a,b-1)\ar[uu] &&\scriptstyle&& 	\scriptstyle &&\scriptstyle\ar[uu]&& &&\scriptstyle\cdots\ar[uu]&&\\	
&&& \\	
&&\scriptstyle\cdots\ar[uu] && && &&\scriptstyle\ar[uu] &&  &&\scriptstyle\cdots\ar[uu]\\	
&&	& \\
&&\scriptstyle (2,a,j)\ar[rr] \ar[uu]&& \scriptstyle\ar[rr] &&\scriptstyle\ar[rr]  && \scriptstyle\old{\Omega^{-1}_{\widehat{B}}\overline{\varphi}(X)=(2,i,j)}\ar[rr]\ar[uu] &&\scriptstyle\cdots\ar[rr]&&\scriptstyle(2,a+p,j)\ar[uu]\\
&&& \\
&&\scriptstyle\cdots \ar[uu] && \scriptstyle && \scriptstyle&&\scriptstyle \ar[uu] &&\scriptstyle &&\scriptstyle\cdots\ar[uu]\\ 
&&\\
&&\scriptstyle(2,a,b-q) \ar[uu] \ar[rr]&&\scriptstyle\cdots\ar[rr]&&\scriptstyle\cdots\ar[rr]&&\scriptstyle(2,i,b-q)\ar[uu]\ar[rr] &&\scriptstyle\cdots\ar[rr]&&\scriptstyle(2,a+p,b-q) \ar[uu]
}\] 
\end{small}
\medskip	

Let $R$ be a finite dimensional algebra and $X$ a finitely generated $R$-module. We introduce some notations as follows.  For a family of objects $\mathcal{X}$ in $R$-$\stmod$, we define 
\begin{equation}
\begin{aligned}
&\RS_{R}(\mathcal{X})\colon=\{Z\in R{\text-}\stmod \mid \StHom_{R}(X,Z)\neq0, \exists  X\in\mathcal{X}\}, \nonumber\\  &\LS_{R}(\mathcal{X})\colon=\{Z\in R{\text-}\stmod\mid \StHom_{R}(Z,X)\neq0, \exists  X\in\mathcal{X}\}, \nonumber\\&{^{\bot}\mathcal{X}^{\bot}}\colon=\{Z\in R{\text-}\stmod \mid \StHom_{R}(X,Z)=0, \StHom_{R}(Z,X)=0,\forall X\in\mathcal{X}\}.\nonumber
\end{aligned}
\end{equation}
$\RS_{R}\mathcal{X}$ is called the {\bf right support of $\mathcal{X}$},  $\LS_{R}\mathcal{X}$  the {\bf left support of $\mathcal{X}$} and ${^{\bot}\mathcal{X}^{\bot}}$  the {\bf stable bi-perpendicular category of $\mathcal{X}$}. It is clear that  ${^{\bot}\mathcal{X}^{\bot}}=R$-$\stmod\backslash(\RS_{R}\mathcal{X}\cup\LS_{R}\mathcal{X})$. We still use the notations $\RS_{R} X$, $\LS_{R} X$ and ${^{\bot}X^{\bot}}$ for an object $X$ in $R$-$\stmod$.

Take a non-projective object $X$ in $R$-$\m$,  we define

\begin{equation}\label{Predecessor-1}
\begin{aligned}
\mathcal{P}_{R}(X):=\LS_{R}(X)\cap\{&M\in R\text{-}\ind\mid \text{there is a non-zero compositions of finitely many irreducible}\\
&\text{ maps from M to X in the AR-quiver $\Gamma_{R}$}\}.
\end{aligned}
\end{equation}
Dually,  
\begin{equation}\label{Sucessor-1}
\begin{aligned}
\mathcal{S}_{R}(X):=\RS_{R}(X)\cap\{&N\in R\text{-}\ind\mid \text{there is a non-zero  compositions of finitely many irreducible}\\
&\text{ maps from X to N in the AR-quiver $\Gamma_{R}$}\}.
\end{aligned}
\end{equation}
Note that  every object on $\mathcal{P}_{R}(X)$ is a predecessor of $X$ and every object on $\mathcal{S}_{R}(X)$ is a successor of $X$ in the AR-quiver $\Gamma_{R}$ of $R$.  

\begin{Cor}\label{varphi-biperpen}
Let $\widehat{B}$, $\overline{\varphi}$ and $X$ be as in the  Proposition \ref{action-of-varphi}. Then $\overline{\varphi}(X)$ \rm{(}resp. $\overline{\varphi}^{-1}(X)$\rm{)} is contained in $^{\bot}{\Omega_{\widehat{B}}^{-1}(X)}^{\bot}$\rm{ (}resp. $^{\bot}{\Omega_{\widehat{B}}(X)}^{\bot}$\rm{)} in $\widehat{B}$-$\stmod$. 
\end{Cor}

Let $A$ be a 1-domestic Brauer graph algebra. By Section  1 and 2, $A$ is of the form $\widehat{B}/<\varphi>$ and there is  a canonical Galois covering $F:\widehat{B}\rightarrow A$ induced by $\varphi$. The  functor $\overline{F_{\lambda}}:\widehat{B}$-$\stmod\rightarrow A$-$\stmod$ induced by push down  functor $F_{\lambda}$ is a covering functor. We present the proposition as follows.

\begin{Prop}\label{supp-bipendic}
Let $A$ be a 1-domestic Brauer graph algebra and $X$ an indecomposable non-projective  object in a stable  connected component $\mathcal{C}$  of $\Gamma_{A}$.
\begin{enumerate}[$(1)$]
\item The intersection between left support of $X$ and $\mathcal{C}$ states as follows. 
\begin{equation}
\begin{aligned}
\LS_{A}(X)\cap\mathcal{C} =&\overline{F_{\lambda}}(\LS_{\widehat{B}}(X))\cap\mathcal{C}=\overline{F_{\lambda}}(\RS_{\widehat{B}}(\nu^{-1}_{\widehat{B}}\Omega^{-1}_{\widehat{B}}(X)))\cap\mathcal{C}\\
=&\overline{F_{\lambda}}(\mathcal{P}_{\widehat{B}}(X)\cup\mathcal{S}_{\widehat{B}}(\nu^{-1}_{\widehat{B}}\Omega_{\widehat{B}}^{-1}(X)))=\mathcal{P}_{A}(X)\cup\mathcal{S}_{A}(\Omega^{-1}(X)).
\end{aligned}
\end{equation}
\item The intersection between right support of $X$ and $\mathcal{C}$ states as follows.  
\begin{equation}
\begin{aligned}
\RS_{A}(X)\cap\mathcal{C} &=\overline{F_{\lambda}}(\RS_{\widehat{B}}(X))\cap\mathcal{C}=\overline{F_{\lambda}}(\LS_{\widehat{B}}(\nu_{\widehat{B}}\Omega_{\widehat{B}}(X)))\cap\mathcal{C}\\
&=\overline{F_{\lambda}}(\mathcal{S}_{\widehat{B}}(X)\cup\mathcal{P}_{\widehat{B}}(\nu_{\widehat{B}}\Omega_{\widehat{B}}(X)))=\mathcal{S}_{A}(X)\cup\mathcal{P}_{A}(\Omega(X)).
\end{aligned}
\end{equation}

\item The intersection of bi-perpendicular category of $X$ and $\mathcal{C}$ states as follows.   \[^{\bot}X^{\bot}\cap\mathcal{C}=\mathcal{C}\setminus(\mathcal{P}_{A}(X)\cup\mathcal{S}_{A}(X)\cup\mathcal{S}_{A}(\Omega^{-1}(X))\cup\mathcal{P}_{A}(\Omega(X))).\]
\end{enumerate}
\end{Prop}
\begin{proof}
We prove only conclusion (1), since (2) is a dual of (1) and (3) is a direct consequence of (1) and (2). We still use $X$ to denote a  preimage of $X$ under covering functor $F_{\lambda}$ and we assume that $\varSigma$ is the union  $\Omega^{-1}\mathcal{C}\cup\mathcal{C}\cup\Omega\mathcal{C}$.
	
\noindent (1) The second equation follows directly from Serre duality and  the fourth equation holds by the definition of $\mathcal{P}(A)$ and $\mathcal{S}(A)$ since   $\overline{F_{\lambda}}$ is a covering functor.  It suffices to show the first and the third equations holds.
If $Z\in\LS_{A}(X)\cap\mathcal{C}$, then $Z\in\mathcal{C}$ and  $\StHom_{A}(Z,X)\ncong0$. By covering theory, 
\[\StHom_{A}(Z,X)\cong\bigoplus_{i\in\mathbb{Z}}\StHom_{\widehat{B}}(\overline{\varphi}^{i}(Z),X)\ncong0.\]
Therefore there is an integer $j$ such that $\StHom_{\widehat{B}}(\overline{\varphi}^{j}(Z),X)\ncong0$. Hence  $\overline{\varphi}^{j}(Z)\in\LS_{\widehat{B}}(X)$. Since  $\overline{F_{\lambda}}(\overline{\varphi}^{i}(Z))\cong Z$ in $A$-$\stmod$, $Z\in \overline{F_{\lambda}}(\LS_{\widehat{B}}(X))\cap\mathcal{C}$. 
Conversely, if $Z\in \overline{F_{\lambda}}(\LS_{\widehat{B}}(X))\cap\mathcal{C}$, 
then there is a $W$ in $\widehat{B}$-$\stmod$ such that $Z\cong\overline{F_{\lambda}}(W)$ and $W\in\LS_{\widehat{B}}(X)$. By covering theory, 
\[\StHom_{A}(Z,X)\cong\bigoplus_{i\in\mathbb{Z}}\StHom_{\widehat{B}}(\overline{\varphi}^{i}(Z),X)\cong\bigoplus_{i\in\mathbb{Z}}\StHom_{\widehat{B}}(\overline{\varphi}^{i}(W),X)\ncong0.\]
Therefore $Z\in\LS_{A}(X)\cap\mathcal{C}.$ Thus the first equation holds, that is,  $\LS_{A}(X)\cap\mathcal{C}=\overline{F_{\lambda}}(\LS_{\widehat{B}}(X))\cap\mathcal{C}$. 
	
 By Proposition \ref{repetitive-stable generalized}, every connected component of $\Gamma_{\widehat{B}}$ is stable generalized standard.    By Serre duality, \[\LS_{\widehat{B}}(X)\cap{\varSigma}=\mathcal{P}_{\widehat{B}}(X)\cup\mathcal{S}_{\widehat{B}}(\nu^{-1}_{\widehat{B}}\Omega_{\widehat{B}}^{-1}(X)),\ \ \ \ \RS_{\widehat{B}}(X)\cap{\varSigma}=\mathcal{S}_{\widehat{B}}(X)\cup\mathcal{P}_{\widehat{B}}(\nu_{\widehat{B}}\Omega_{\widehat{B}}(X)).\]
Therefore \[\overline{F_{\lambda}}(\LS_{\widehat{B}}(X))\cap\mathcal{C}=\overline{F_{\lambda}}(\LS_{\widehat{B}}(X)\cap\varSigma)=\overline{F_{\lambda}}(\mathcal{P}_{\widehat{B}}(X)\cup\mathcal{S}_{\widehat{B}}(\nu^{-1}_{\widehat{B}}\Omega_{\widehat{B}}^{-1}(X))),\]
\[\overline{F_{\lambda}}(\RS_{\widehat{B}}(X))\cap\mathcal{C}=\overline{F_{\lambda}}(\RS_{\widehat{B}}(X)\cap\varSigma)=\overline{F_{\lambda}}(\mathcal{S}_{\widehat{B}}(X)\cup\mathcal{R}_{\widehat{B}}(\nu_{\widehat{B}}\Omega_{\widehat{B}}(X))).\]
Thus (1) holds. 
\end{proof}

\begin{Lem}\label{stable-Euclidean}
Let $A$ be a 1-domestic Brauer graph algebra and $\mathcal{C}$ the stable Euclidean component.  Then every object in $\mathcal{C}$ is a stable brick in $A$-$\stmod$.
\end{Lem}
\begin{proof}
By Theorem \ref{symmetric-algebras-of-Euclidean-type-1} $A$ is isomorphic to the algebra $\widehat{B}/<\varphi>$ and there is a canonical Galois covering 
$F:\widehat{B}\rightarrow A \cong\widehat{B}/<\varphi>,$ where $\varphi$ is a square root of  the Nakayama automorphism $\nu_{\widehat{B}}$ is of $\widehat{B}$. 
Since $\widehat{B}$ is locally support-finite, the push down functor 
\[F_{\lambda}: \widehat{B}\text{-} \m\rightarrow A\text{-} \m\]
induced by $F$ is dense.
Therefore we have the following isomorphism induced by $F_{\lambda}$.
\[\Hom_{A}(M,M)\cong\bigoplus_{ i\in\mathbb{Z}}\Hom_{\widehat{B}}(\varphi^{i}(M),M).\]
For an indecomposable non-projective module $M$ in $A$-$\m$, we shall use the same notations  in $\widehat{B}$-$\m$ if there is no confusion. 

Combining the action of $\nu_{\widehat{B}}$ in the AR-quiver of $\widehat{B}$ and \cite[Lemma 5.1]{SY1}, we have $\Hom_{\widehat{B}}(\varphi^{i}(M),M)=0$ for any $i\neq -1, 0$ and $1$.  Hence  \[\Hom_{A}(M,M)=\Hom_{\widehat{B}}(\varphi^{-1}(M),M)\bigoplus\Hom_{\widehat{B}}(M,M)\bigoplus\Hom_{\widehat{B}}(\varphi(M),M).\]
By Serre duality, \[\StHom_{\widehat{B}}(\overline{\varphi}^{-1}(M),M)\cong D\StHom_{\widehat{B}}(M,\nu_{\widehat{B}}\Omega_{\widehat{B}}\overline{\varphi}^{-1}(M))=D\StHom_{\widehat{B}}(M,\Omega_{\widehat{B}}\overline{\varphi}(M))\cong D\StHom_{\widehat{B}}(\Omega_{\widehat{B}}^{-1}(M),\overline{\varphi}(M)).\]
By Corollary \ref{varphi-biperpen},  $\overline{\varphi}(X)\in{^{\bot}{\Omega_{\widehat{B}}^{-1}(X)}^{\bot}}$.
Therefore  $\StHom_{\widehat{B}}(\overline{\varphi}^{-1}(M),M)\cong D\StHom_{\widehat{B}}(\Omega_{\widehat{B}}^{-1}(M),\overline{\varphi}(M))=0.$
It can be similarly proved that  $\StHom_{\widehat{B}}(\overline{\varphi}(M),M)=0$.  It follows from Proposition \ref{repetitive-stable generalized} that every object in $_{s}\Gamma_{\widehat{B}}$ is a stable brick in $\widehat{B}$-$\stmod$. 
Thus $\StHom_{A}(M,M)\cong\StHom_{\widehat{B}}(M,M)\cong k.$
\end{proof}

For the rest of this section, we consider the action of $\overline{\varphi}$ in quasi-tubes. Let $R$ be a representation-infinite algebra and $\mathcal{C}$ a quasi-tube  of rank $n\geq 1$ of the AR-quiver of $R$.
If $X$ is an indecomposable non-projective object lying at the end of $\mathcal{C}$, that is, a {\bf quasi-simple} of $\mathcal{C}$,
then for any natural number $r\geq1$, there is a unique infinite sectional path starting at $X$
\[
X = X(0)\rightarrow X(1)\rightarrow \cdots \rightarrow X(r-1)\rightarrow X(r)\rightarrow \cdots.
\]
A non-projective object in $\mathcal{C}$ is  of  {\bf quasi-length $r$} if it is of the form $X(r)$ for some quasi-simple $X$ of $\mathcal{C}$. Note that  the quasi-length of a quasi-simple is $0$ under our assumption.
\begin{Lem}\label{quasi-tube-varphi}
Let $B$ be an exceptional  Euclidean algebra of type $\widetilde{A}_{m}$ and $\varphi$ a square root of Nakayama automorphism $\nu_{\widehat{B}}$.  
\begin{enumerate}
\item  If $\overline{\varphi}$  preserves the orientation of irreducible maps in the Euclidean components, then $\overline{\varphi}(_{s}P_{i})=\Omega^{-1}_{\widehat{B}}(_{s}P_{i})$ and  $\overline{\varphi}(_{s}Q_{i})=\Omega^{-1}_{\widehat{B}}(_{s}Q_{i})$ for any quasi-tubes $_{s}P_{i}$ and $_{s}Q_{i}$ in the stable AR-quiver $_{s}\Gamma_{\widehat{B}}$ of $\widehat{B}$.	
\item If $\overline{\varphi}$ inverses  the orientation of irreducible maps in the Euclidean components, then $\overline{\varphi}(_{s}P_{i})=\Omega^{-1}(_{s}Q_{i})$ and  $\overline{\varphi}(_{s}Q_{i})=\Omega^{-1}(_{s}P_{i})$ for any quasi-tubes $_{s}P_{i}$ and $_{s}Q_{i}$  in the stable AR-quiver $_{s}\Gamma_{\widehat{B}}$ of $\widehat{B}$.
\end{enumerate}
\end{Lem}

\begin{Lem}\label{quasi-tube-brick}
Let $A$ be a representation-infinite symmetric algebra and $\mathcal{C}$ a quasi-tube of rank $p$ satisfying  $\Omega(\mathcal{C})=\mathcal{C}$.  If $\Omega(X)\cong \tau^{r}(X)$ for a quasi-simple $X$ in $\mathcal{C}$ and a positive integer $r$, then any module with quasi-length  larger than or equal to $p-r-1$ is not a stable brick in $A$-$\stmod$.
\end{Lem}
\begin{proof} 
Take an  object $Y(s)$ in $_{s}\mathcal{C}$.  By  \cite[2.2]{EK}, 
\begin{equation} \label{stable-brick-equation-1}
\StHom_{A}(Y(s),Y(s))\cong \bigoplus_{t=-s}^{t=0}\StHom_{A}(Y(s),\tau^{t}Y).
\end{equation}
It is clear that  $\StHom_{A}(Y(s),\tau^{-s}Y)\ncong0$  since there is a sectional path from $Y(s)$ to $\tau^{-s}Y$. Since $\Omega(X)\cong \tau^{r}(X)$ for the quasi-simple $X$ in $\mathcal{C}$ and a positive integer $r$,  $\Omega(W)\cong \tau^{r}(W)$ for any quasi-simple $W$ in $\mathcal{C}$. If $s\geq p-r-1$, then $\Omega(Y)$ is  in the set $\{\tau^{t}Y\mid -s\leq t\leq 0\}$. Thus  $\StHom_{A}(Y(s),\Omega(Y))$ is a direct summand at the right side of  the equation (\ref{stable-brick-equation-1}). By Serre duality,  \[\StHom_{A}(Y(s),\Omega(Y))\cong \D\StHom_{A}(\Omega(Y), \nu\Omega(Y(s)))\cong\D\StHom_{A}(Y, Y(s)).\] It is not isomorphic to zero.  Thus $\dim_{k} \StHom_{A}(Y(s),Y(s))\geq 2$ and $Y(s)$ is not a stable brick in $A$-$\stmod$.
\end{proof}

\begin{Lem}\label{quasi-tube-brick-2}
Let $A$ be a 1-domestic Brauer graph algebra and $\mathcal{C}$ a quasi-tube of rank $p$ satisfying  $\Omega(\mathcal{C})=\mathcal{C}$.  If $\Omega(X)\cong \tau^{r}(X)$ for a quasi-simple $X$ in $\mathcal{C}$ and a positive integer $r$, then 
\begin{enumerate}
\item Any module with quasi-length  less than $p-r-1$ is a stable brick in $A$-$\stmod$.
\item Any module with quasi-length  larger than or equal to $p-r-1$ is not a stable brick in $A$-$\stmod$.
\end{enumerate}
\end{Lem}
\begin{proof}
It is a direct consequence  of  Lemma \ref{quasi-tube-brick} and Proposition 2.2 in \cite{Ap}. 
\end{proof}

\begin{Lem}\label{quasi-tube-brick-3}
Let $A$ be a 1-domestic Brauer graph algebra and $\mathcal{C}$ a quasi-tube of rank $p$ satisfying $\Omega(\mathcal{C})\neq\mathcal{C}$.  Then any object with quasi-length  less than $p-1$ is a stable brick in $A$-$\stmod$.
\end{Lem}
\begin{proof}
It follows directly from covering theory, Proposition \ref{repetitive-stable generalized} and the equation (\ref{stable-brick-equation-1}).
\end{proof}

\section{The construction of simple-minded systems}
Let $A$ be a 1-domestic Brauer graph algebra. In this section, we shall provide a construction of simple-minded systems in $A$-$\stmod$. By Corollary \ref{sms-BGA}, we need present a construction of maximal orthogonal systems containing at least one object of the stable Euclidean component.  We assume that $A=\widehat{B}/<\varphi>$ and the stable AR-quiver $_{s}\Gamma_{A}$ of $A$ consists of a stable Euclidean component $_{s}\Gamma$ of the form $\mathbb{Z}A_{p,q}$, one stable quasi-tube $_{s}P$ of rank $p$, one stable quasi-tube $_{s}Q$ of rank $q$ and infinitely many homogeneous tubes. By \cite[Theorem 1.3]{CLZ}, any module of quasi-length larger than or equal to the rank of quasi-tube can not be  in a simple-minded system, in particular, any  module in a homogeneous tube is not in a simple-minded  system. Therefore we will not consider any homogeneous tubes in this section.  

Recall that we use $(0,i,j)$ to  represent  an object in the stable Euclidean component $_{s}\Gamma$, $\tau(0,i,j)=(0,i-1,j-1)$ and $(0,i,j)=(0,i-p\ell,j+qk)$ for any integers $\ell$ and $k$. We  use $P(0,i,j)$ (resp. $Q(0,i,j)$) to  represent  an object in the  stable quasi-tube $_{s}P$ (resp. $_{s}Q$). We   assume that 
\[\tau P(0,x,y)=P(0,x-1,y),\  \text{and}\ \tau Q(0,x,y)=Q(0,x-1,y).\] 
We  also assume that 
\[\Omega^{-1}((0,i,j))=(1,i,j), \ \Omega^{-1}(P(0,i,j))=P(1,i,j)\ \text{and }\ \Omega^{-1}(Q(0,i,j))=Q(1,i,j).\]

We now provide a construction of  maximal orthogonal systems containing at least one object for the stable Euclidean component. Take an object $X=(0,a,b)$ in the stable Euclidean component  $_{s}\Gamma$. By Proposition \ref{supp-bipendic}, 
\begin{equation}\label{bi-pen-1}
\begin{aligned}
^{\bot}X^{\bot}\bigcap{_{s}\Gamma}={_{s}\Gamma}\setminus\left(\mathcal{P}_{A}(X)\bigcup\mathcal{S}_{A}(X)\bigcup\mathcal{S}_{A}(\Omega^{-1}(X))\bigcup\mathcal{P}_{A}(\Omega(X))\right).
\end{aligned}
\end{equation}
By covering theory, the position of $\Omega(X)$ and $\Omega^{-1}(X)$ in the stable AR-quiver $_{s}\Gamma$ are determined by the action of $\overline{\varphi}$
and they are contained in a finite area related to $X$ in $_{s}\Gamma$ by Proposition \ref{action-of-varphi}. 
Thus $^{\bot}X^{\bot}\cap{_{s}\Gamma}$ is  determined.  $^{\bot}X^{\bot}\cap({_{s}Q}\cup{_{s}P})$ can be determined similarly. 

Take an object $Y(s)$ in stable quasi-tube $_{s}Q$.  By Proposition \ref{supp-bipendic}, 
\begin{equation}\label{bi-pen-2}
\begin{aligned}
^{\bot}Y(s)^{\bot}\bigcap\left({_{s}Q}\bigcup{_{s}P}\right)=&\left({_{s}Q}\bigcup{_{s}P}\right)\setminus\left(\mathcal{P}_{A}(Y(s))\bigcup\mathcal{S}_{A}(Y(s))\bigcup\mathcal{S}_{A}(\Omega^{-1}(Y(s)))\bigcup\mathcal{P}_{A}(\Omega(Y(s)))\right)
\end{aligned}
\end{equation}
According to  Erdmann and Kerner  in \cite[2.2]{EK}, for any object $W$ in the stable Euclidean component  $_{s}\Gamma$, 
\begin{equation}\label{Qua-Euc-0}
\StHom_{A}(W,Y(s))\cong \bigoplus_{t=-s}^{t=0}\StHom_{A}(W,\tau^{t}Y),\  \ \ 
\StHom_{A}(Y(s),W)\cong \bigoplus_{t=-s}^{t=0}\StHom_{A}(\tau^{t}Y,W).
\end{equation}
Thus 
\begin{equation} \label{Qua-Euc-1}
\begin{aligned}
\LS_{A}(Y(s))\bigcap{_{s}\Gamma}=&\bigcup_{t=-s}^{0}\LS_{A}(\tau^{t}Y)\bigcap{_{s}\Gamma},\\
\RS_{A}(Y(s))\bigcap{_{s}\Gamma}=&\bigcup_{t=-s}^{0}\RS_{A}(\tau^{t}Y)\bigcap{_{s}\Gamma}.
\end{aligned}
\end{equation}
Then
\begin{equation} \label{Qua-Euc-3}
\begin{aligned}
^{\bot}Y(s)^{\bot}\bigcap{_{s}\Gamma}=&{_{s}\Gamma}\backslash\left(\left(\LS_{A}(Y(s))\bigcup\RS_{A}(Y(s))\right)\bigcap{_{s}\Gamma}\right)\\
=&{_{s}\Gamma}\backslash\left(\left(\bigcup_{t=-s}^{0}\LS_{A}(\tau^{t}Y)\bigcup_{t=-s}^{0}\RS_{A}(\tau^{t}Y)\right)\bigcap{_{s}\Gamma}\right).
\end{aligned}
\end{equation}
Since $\tau^{t}Y$ is quasi-simple for each $t$ and  $\RS_{A}(\tau^{t}Y)\cap{_{s}\Gamma}$ is known (please refer to \cite[Section 3.3]{Z1}), 
$^{\bot}Y(s)^{\bot}\cap{_{s}\Gamma}$ is determined.

Take an orthogonal system $\mathcal{X}$ in the stable Euclidean component $_{s}\Gamma$.  By equations (\ref{bi-pen-1}) and (\ref{bi-pen-2}), we can determine the set $^{\bot}{\mathcal{X}}^{\bot}\cap({_{s}\Gamma}\cup{_{s}Q}\cup{_{s}P})$. Note that $^{\bot}{\mathcal{X}}^{\bot}\cap({_{s}\Gamma}\cup{_{s}Q}\cup{_{s}P})$ is a finite set. We add one object $X_{1}$ of $^{\bot}{\mathcal{X}}^{\bot}\cap({_{s}\Gamma}\cup{_{s}Q}\cup{_{s}P})\cap\mathcal{STB}_{A}$ to $\mathcal{X}$ such that $\mathcal{X}_{2}=\mathcal{X}\cup\{X_{1}\}$,  where $\mathcal{STB}_{A}$ is the set of stable bricks in $A$-$\stmod$. By equations (\ref{bi-pen-1}), (\ref{bi-pen-2}) and (\ref{Qua-Euc-3}), we may determine $^{\bot}{\mathcal{X}_{2}}^{\bot}\cap({_{s}\Gamma}\cup{_{s}Q}\cup{_{s}P})\cap\mathcal{STB}_{A}$, which is  a subset of $^{\bot}{\mathcal{X}}^{\bot}\cap({_{s}\Gamma}\cup{_{s}Q}\cup{_{s}P})\cap\mathcal{STB}_{A}$. Proceeding this process, within  finitely many steps, we may construct a maximal orthogonal system $\mathcal{X}_{m}$ containing  $\mathcal{X}$ for some positive integer $m$.  By Corollary \ref{sms-BGA}, $\mathcal{X}_{m}$ is  a simple-minded systems containing orthogonal system $\mathcal{X}$ in $A$-$\stmod$. According to the action of $\overline{\varphi}$, there are two cases to be considered. 

\subsection{The case that $\overline{\varphi}$ preserves the orientation of the irreducible maps in the Euclidean components.}
By Lemma \ref{quasi-tube-varphi},  $\Omega(_{s}P)={_{s}P}$ and $\Omega(_{s}Q)={_{s}Q}$. 
Let $X$ be an object  in the stable Euclidean component $_{s}\Gamma$.
By Proposition \ref{action-of-varphi}, $\Omega(X)$ is determined by $p$ and $q$, specifically, if $X=(0,a,b)$, then  
\begin{equation}
\begin{aligned}\label{eq-00}
\Omega(X)=(-1,a, b)=\left(0,a+\dfrac{p-1}{2}, b-\dfrac{q+1}{2}\right).
\end{aligned}
\end{equation}
It follows that 
\begin{equation}
\begin{aligned}\label{eq-01}
\Omega=\tau^{\dfrac{p+1}{2}}\left(\text{resp.}\  \Omega=\tau^{\dfrac{q+1}{2}}\right)
\end{aligned}
\end{equation}
in the stable quasi-tube $_{s}P$ (resp. $_{s}Q$).

By Proposition \ref{supp-bipendic},
\begin{equation}
\begin{aligned}\label{eq-0}
\LS_{A}(X)\bigcap{_{s}\Gamma}=\mathcal{P}_{A}(X)\bigcup\mathcal{S}_{A}(\Omega^{-1}(X)),\ \ 
\RS_{A}(X)\bigcap{_{s}\Gamma}=\mathcal{S}_{A}(X)\bigcup\mathcal{P}_{A}(\Omega(X)).
\end{aligned}
\end{equation}
Thus
\begin{equation}
\begin{aligned}\label{bi-perp-1}
^{\bot}X^{\bot}\bigcap{_{s}\Gamma}&={_{s}\Gamma}\setminus\left(\mathcal{P}_{A}(X)\bigcup\mathcal{S}_{A}(X)\bigcup\mathcal{S}_{A}(\Omega^{-1}(X))\bigcup\mathcal{P}_{A}(\Omega(X))\right)\\
&=\left\{(0,a+i,b-j)\mid 1\leq i\leq\dfrac{p-1}{2}, 1\leq j\leq\dfrac{q+3}{2}\right\}\\
&\bigcup\left\{(0,a+i,b-j)\mid \dfrac{p+1}{2}\leq i\leq p-1, \dfrac{q+1}{2}\leq j\leq q-1\right\}.
\end{aligned}
\end{equation}

Now we determine $^{\bot}X^{\bot}\cap({_{s}Q}\cup{_{s}P})$. 
Consider irreducible maps ended at $X=(0,a,b)$ in $\Gamma$ as follows. 
\[\alpha_{1}:(0,a,b-1)\rightarrow(0,a,b),\ \  \beta_{1}: (0,a-1,b)\rightarrow(0,a,b).\]
Extending $\alpha_{1}$ and $\beta_{1}$ to triangles in $A$-$\stmod$ as follows.
\begin{equation}
\begin{aligned}\label{bi-perp-21}
&(0,a,b-1)\xrightarrow{\alpha_{1}}(0,a,b)\xrightarrow{}Z_{0}\xrightarrow{} (0,a,b-1)[1],\\
&(0,a-1,b)\xrightarrow{\beta_{1}}(0,a,b)\xrightarrow{}Z_{1}\xrightarrow{} (0,a-1,b)[1].
\end{aligned}
\end{equation}
Dually,  there are precise two irreducible maps started at $X=(0,a,b)$	in ${_{s}\Gamma}$. 
\[\alpha_{2}:(0,a,b)\rightarrow(0,a,b+1),\ \  \beta_{2}:(0,a,b)\rightarrow(0,a+1,b).\]
Extending $\alpha_{2}$ and $\beta_{2}$ to triangles as follows.
\begin{equation}
\begin{aligned}\label{bi-perp-31}
&Z'_{0}\xrightarrow{}(0,a,b)\xrightarrow{\alpha_{2}}(0,a,b+1)\xrightarrow{} Z'_{0}[1],\\
&Z'_{1}\xrightarrow{}(0,a,b)\xrightarrow{\beta_{2}}(0,a+1,b)\xrightarrow{}Z'_{1}[1].
\end{aligned}
\end{equation}
By \cite[Proposition 3.4]{Z}, both $Z_{0}$ and $Z_{1}$ (resp. $Z'_{0}$ and $Z'_{1}$) are quasi-simples. Note that one of them is in a quasi-tube of rank $p$ and the other one is in a quasi-tube of rank $q$. Without loss of generality, we assume that $Z'_{0}$ (resp. $Z'_{1}$) is  in the quasi-tube $Q$ (resp. $P$) of rank $q$ (resp. $p$). We denote $Z'_{0}$ (resp. $Z'_{1}$) by $Q(0,0,0)$ (resp. $P(0,0,0)$).
By equation (\ref{eq-01}),
\[Z_{0}=\Omega(Z'_{0})=Q(-1,0,0)=Q\left(0,\dfrac{q+1}{2},0\right),\ \  Z_{1}=\Omega(Z'_{1})=P(-1,0,0)=P\left(0,\dfrac{p+1}{2},0\right).\]

\begin{Def}
Let $A$ be a self-injective algebra and $X=(j,k)$ an object in a stable quasi-tube. The {\bf wing} of $X$ is the set of objects in the quasi-tube given by
\begin{align*}
W_{X} &:= \{(m,\ell) \mid m\leq j,\;\; j+k\leq m+\ell\}. 
\end{align*}
\end{Def}

\begin{Def}\label{triangle-areas}
Let $A$ be a self-injective algebra and $(a,b)$ an object  in a stable  quasi-tube. Take the set 
\[\bigtriangleup_{(a,b)}\colon=\{(i,j)\mid a\leq i\leq a+b, a\leq i+j\leq a+b\}.\]
$\bigtriangleup_{(a,b)}$  is called {\bf triangle area of object $(a,b)$} and we call the number $b$ the {\bf height} of $\bigtriangleup_{(a,b)}.$
\end{Def}

 By equation (\ref{Qua-Euc-0}), 
\begin{equation}\label{eq-1}	
\begin{aligned}
\LS_{A}(0,a,b)\bigcap{_{s}Q}=W_{Q(0,0,0)},\ \ \ \ 
\RS_{A}(0,a,b)\bigcap{_{s}Q}=W_{Q(1,0,0)}=W_{Q\left(0,\dfrac{q+1}{2},0\right)};\\ 
\LS_{A}(0,a,b)\bigcap{_{s}P}=W_{P(0,0,0)}, \ \ \ \ \RS_{A}(0,a,b)\bigcap{_{s}P}=W_{P(1,0,0)}=W_{P\left(0,\dfrac{p+1}{2},0\right)}.
\end{aligned}	
\end{equation}

Thus 

\begin{equation}
\begin{aligned}\label{eq-3}	
\LS_{A}(0,a,b)\bigcap\left({_{s}Q}\bigcup{_{s}P}\right)&=W_{Q(0,0,0)}\bigcup W_{P(0,0,0)},\\
\RS_{A}(0,a,b)\bigcap\left({_{s}Q}\bigcup{_{s}P}\right)&=W_{Q\left(0,\dfrac{q+1}{2},0\right)}\bigcup W_{P\left(0,\dfrac{p+1}{2},0\right)}.
\end{aligned}	
\end{equation}

Thus 
\begin{equation}
\begin{aligned}\label{bi-perp-2}
^{\bot}X^{\bot}\bigcap{_{s}Q}=&{_{s}Q}\setminus\left(W_{Q(0,0,0)}\bigcup W_{Q\left(0,\dfrac{q-1}{2},0\right)}\right)\\
=&\left\{Q(0,j,k)\mid 1\leq j\leq\dfrac{q-3}{2}, 1\leq j+k\leq\dfrac{q-3}{2}\right\}\\
&\bigcup\left\{Q(0,j,k)\mid \dfrac{q+1}{2}\leq j\leq q-1, \dfrac{q+1}{2}\leq j+k\leq q-1\right\}\\
=&\bigtriangleup_{Q\left(0,1,\dfrac{q-5}{2}\right)}\bigcup\bigtriangleup_{Q\left(0,\dfrac{q+1}{2},\dfrac{q-3}{2}\right)}.
\end{aligned}
\end{equation}

\begin{equation}
\begin{aligned}\label{bi-perp-3}
^{\bot}X^{\bot}\bigcap{_{s}P}=&{_{s}P}\setminus\left(W_{P(0,0,0)}\bigcup W_{P\left(0,\dfrac{p-1}{2},0\right)}\right)\\
=&\left\{Q(0,j,k)\mid 1\leq j\leq\dfrac{p-3}{2}, 1\leq j+k\leq\dfrac{p-3}{2}\right\}\\
&\bigcup\left\{Q(0,j,k)\mid \dfrac{p+1}{2}\leq j\leq p-1, \dfrac{p+1}{2}\leq j+k\leq p-1\right\}\\
=&\bigtriangleup_{P\left(0,1,\dfrac{p-3}{2}\right)}\bigcup\bigtriangleup_{P\left(0,\dfrac{p+1}{2},\dfrac{p-3}{2}\right)}.
\end{aligned}
\end{equation}

By equations (\ref{bi-perp-2}) and (\ref{bi-perp-3}), we have 
\begin{equation}
\begin{aligned}\label{bi-perp-4}
^{\bot}X^{\bot}\bigcap\left({_{s}Q}\bigcup{_{s}P}\right)&=\left({_{s}Q}\bigcup{_{s}P}\right)\setminus\left(W_{Q(0,0,0)}\bigcup W_{P(0,0,0)}\bigcup W_{Q\left(0,\dfrac{q-1}{2},0\right)}\bigcup W_{P\left(0,\dfrac{p-1}{2},0\right)}\right)\\
&=\bigtriangleup_{Q\left(0,1,\dfrac{q-5}{2}\right)}\bigcup\bigtriangleup_{Q\left(0,\dfrac{q+1}{2},\dfrac{q-3}{2}\right)}\bigcup\bigtriangleup_{P\left(0,1,\dfrac{p-5}{2}\right)}\bigcup\bigtriangleup_{P\left(0,\dfrac{p+1}{2},\dfrac{p-3}{2}\right)}.
\end{aligned}
\end{equation}

\bigskip

Take an object $Q(0,k,\ell)$ in $_{s}Q$. 
By Proposition \ref{supp-bipendic}, 
\begin{equation}
\begin{aligned}\label{bi-perp-5}
\LS_{A} Q(0,k,\ell)\bigcap{_{s}Q}&=\mathcal{P}_{A}(Q(0,k,\ell))\bigcup\mathcal{S}_{A}(\Omega^{-1}(Q(0,k,\ell)))\\
&=\mathcal{P}_{A}(Q(0,k,\ell))\bigcup\mathcal{S}_{A}\left(\tau^{\dfrac{q-1}{2}}(Q(0,k,\ell))\right)\\
&=\mathcal{P}_{A}(Q(0,k,\ell))\bigcup\mathcal{S}_{A}\left(Q\left(0,k-\dfrac{q-1}{2},\ell\right)\right),
\end{aligned}
\end{equation}

\begin{equation}
\begin{aligned}\label{bi-perp-6}
\RS_{A} Q(0,k,\ell)\bigcap{_{s}Q}&=\mathcal{S}_{A}(Q(0,k,\ell))\bigcup\mathcal{P}_{A}(\Omega(Q(0,k,\ell)))\\
&=\mathcal{S}_{A}(Q(0,k,\ell))\bigcup\mathcal{P}_{A}\left(\tau^{\dfrac{q+1}{2}}(Q(0,k,\ell))\right)\\
&=\mathcal{S}_{A}(Q(0,k,\ell))\bigcup\mathcal{P}_{A}\left(Q\left(0,k-\dfrac{q+1}{2},\ell\right)\right).\\
\end{aligned}
\end{equation}

Thus 
\begin{equation}
\begin{aligned}\label{bi-perp-7}
&{^{\bot}Q(0,k,\ell)^{\bot}}\bigcap{_{s}Q}\\
=&{_{s}Q}\setminus\left(\mathcal{P}_{A}(Q(0,k,\ell))
\bigcup\mathcal{S}_{A}\left(Q\left(0,k-\dfrac{q-1}{2},\ell\right)\right)\bigcup\mathcal{S}_{A}(Q(0,k,\ell))\bigcup\mathcal{P}_{A}\left(Q\left(0,k-\dfrac{q+1}{2},\ell\right)\right)\right).
\end{aligned}
\end{equation}
Similarly, for an object $P(0,r,t)$ in the stable quasi-tube $_{s}P$, we have 
\begin{equation}
\begin{aligned}\label{bi-perp-8}
&{^{\bot}P(0,r,t)^{\bot}}\bigcap{_{s}P}\\
=&{_{s}P}\setminus\left(\mathcal{P}_{A}(P(0,r,t))\bigcup\mathcal{S}_{A}\left(P\left(0,r-\dfrac{p-1}{2},t\right)\right)\bigcup\mathcal{S}_{A}(P(0,r,t))\bigcup\mathcal{P}_{A}\left(P\left(0,r-\dfrac{p+1}{2},t\right)\right)\right).
\end{aligned}
\end{equation}
Moreover,
\begin{equation}
\begin{aligned}\label{bi-perp-9}
{^{\bot}Q(0,k,\ell)^{\bot}}\bigcap{_{s}P}={_{s}P}, \ \ 
{^{\bot}P(0,r,t)^{\bot}}\bigcap{_{s}Q}=	{_{s}Q}.
\end{aligned}
\end{equation}

Take an object $Y(s)$ with the quasi-simple $Y$ in stable quasi-tube $_{s}Q$. 
Consider  triangle 
\begin{equation} \label{qua-Euc-0}
Y[-1]\xrightarrow{}(0,a,b-1)\xrightarrow{\alpha}(0,a,b)\xrightarrow{} Y
\end{equation}
such that morphism $\alpha$ is induced by an irreducible map. 
By Lemma 3.8 in \cite{Z0}, there are triangles 
\begin{equation} \label{qua-Euc-00}
\tau^{-i}Y[-1]\xrightarrow{}(0,a,b+i-1)\xrightarrow{}(0,a,b+i)\xrightarrow{} \tau^{-i}Y
\end{equation}
for each $i\in\mathbb{Z}$.  By equations (\ref{Qua-Euc-0}) and  (\ref{qua-Euc-00}), 
\begin{equation} \label{qua-Euc-2}
\begin{aligned}
\LS_{A}(Y(s))\bigcap{_{s}\Gamma}=&\bigcup_{t=-s}^{0}\LS_{A}(\tau^{t}Y)\bigcap{_{s}\Gamma}\\
=&\{(0,\ell,b+j+kq)\mid  0\leq j\leq s, k, \ell\in\mathbb{Z}\}.
\end{aligned}
\end{equation}
By rotating triangle (\ref{qua-Euc-0}) to the right three times,  we have the following triangle
\begin{equation}\label{qua-Euc-20}
 Y\xrightarrow{}(1,a+1,b)\xrightarrow{}(1,a+1,b+1)\xrightarrow{} Y[1]	
 \end{equation}
 for each $i\in\mathbb{Z}$.
Thus  we have triangles 
\begin{equation}\label{qua-Euc-21}
\tau^{-i}Y\xrightarrow{}(1,a+1,b+i)\xrightarrow{\alpha}(1,a+1,b+i+1)\xrightarrow{} \tau^{-i}Y[1]
\end{equation}
for each $i\in\mathbb{Z}$.
By equations (\ref{Qua-Euc-0}) and (\ref{qua-Euc-21}), 
\begin{equation} \label{qua-Euc-3}
\begin{aligned}
\RS_{A}(Y(s))\bigcap{_{s}\Gamma}=&\bigcup_{t=-s}^{0}\RS_{A}(\tau^{t}Y)\bigcap{_{s}\Gamma}\\
=&\{(1,\ell,b+j+kq)\mid 0\leq j\leq s, k, \ell\in\mathbb{Z}\}\\
=&\left\{(0,\ell,b+j-\dfrac{q+1}{2}+kq)\mid  0\leq j\leq s,k, \ell\in\mathbb{Z}\right\}.
\end{aligned}
\end{equation}
 Note that  $(1,\ell,b+j+kq)=\left(0,\ell,b+j-\dfrac{q+1}{2}+kq\right)$ by the action of $\Omega$ in $_{s}\Gamma$. 
Thus 
\begin{equation} \label{qua-Euc-4}
\begin{aligned}
&^{\bot}Y(s)^{\bot}\bigcap{_{s}\Gamma}\\
=&{_{s}\Gamma}\backslash\left(\left(\LS_{A}(Y(s))\bigcup\LS_{A}(Y(s))\right)\bigcap{_{s}\Gamma}\right)\\
=&{_{s}\Gamma}\backslash\left(\{(0,\ell,b+j+kq)\mid 0\leq j\leq s, k,\ell\in\mathbb{Z}\}\bigcup\left\{(0,\ell,b+j-\dfrac{q+1}{2}+kq)\mid 0\leq j\leq s, k, \ell\in\mathbb{Z}\right\}\right)\\
=&\left\{(0,\ell,b+s+j+kq)\mid 1\leq j\leq b+\dfrac{q-1}{2}, k, \ell\in\mathbb{Z}\right\}\\
&\bigcup\left\{(0,\ell,b+s-j+kq)\mid \dfrac{q-1}{2}\leq j\leq b-s-1, k, \ell\in\mathbb{Z}\right\}.
\end{aligned}
\end{equation}

Take an object $Z(r)$ in the quasi-tube $_{s}P$. 
Consider  triangle 
\begin{equation} \label{qua-Euc-5}
Z[-1]\xrightarrow{}(0,c,d)\xrightarrow{\beta}(0,c+1,d)\xrightarrow{} Z
\end{equation}
such that morphism $\beta$ is induced by an irreducible map.  We have the similar conclusion as follows.
\begin{equation} \label{qua-Euc-6}
\begin{aligned}
\LS_{A}(Z(r))\bigcap{_{s}\Gamma}=&\bigcup_{t=-r}^{0}\LS_{A}(\tau^{t}Z)\bigcap{_{s}\Gamma}\\
=&\{(0,c+j+kp,\ell)\mid  0\leq j\leq r,k, \ell\in\mathbb{Z}\}.
\end{aligned}
\end{equation}

\begin{equation} \label{qua-Euc-7}
\begin{aligned}
\RS_{A}(Z(r))\bigcap{_{s}\Gamma}=&\bigcup_{t=-r}^{0}\RS_{A}(\tau^{t}Z)\bigcap{_{s}\Gamma}\\
=&\{(1,c+j+kp,\ell)\mid  0\leq j\leq r,k, \ell\in\mathbb{Z}\}\\
=&\left\{(0,c+j+\dfrac{p+1}{2}+kp,\ell)\mid  0\leq j\leq r,k, \ell\in\mathbb{Z}\right\}.
\end{aligned}
\end{equation}

\begin{equation} \label{qua-Euc-8}
\begin{aligned}
&^{\bot}Z(r)^{\bot}\bigcap{_{s}\Gamma}\\
=&{_{s}\Gamma}\backslash\left(\LS_{A}(Z(r))\bigcup\RS_{A}(Z(r))\right)\\
=&{_{s}\Gamma}\backslash\left(\{(0,c+j+kp,\ell)\mid 0\leq j\leq r,k, \ell\in\mathbb{Z}\}\bigcup\left\{(0,c+j+\dfrac{p+1}{2}+kp,\ell)\mid 0\leq j\leq r,k, \ell\in\mathbb{Z}\right\}\right)\\
=&\bigcup\left\{(0,c+r+j+kp,\ell)\mid 1\leq j\leq \dfrac{p-1}{2},k, \ell\in\mathbb{Z}\right\}\\
&\bigcup\left\{(0,c+r+j+kp,\ell)\mid \dfrac{p+3}{2}\leq j\leq p-r-1,k, \ell\in\mathbb{Z}\right\}.
\end{aligned}
\end{equation}

By equation (\ref{bi-perp-1}), we may construct any orthogonal system in the Euclidean component $_{s}\Gamma.$
Take an orthogonal systems $\mathcal{X}$ in $_{s}\Gamma$. we may construct a maximal orthogonal system containing  $\mathcal{X}$ in $A$-$\stmod$ by equations  (\ref{bi-perp-4}), (\ref{bi-perp-7})--(\ref{bi-perp-9}), (\ref{qua-Euc-4})  and (\ref{qua-Euc-8}).
We present an example as follows. 

 \begin{Ex}\label{1-domestic-1}
 Consider Brauer graph algebra $A$ with Brauer graph $B$ given by
 
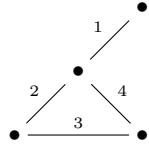
\begin{small}	
 \[\xymatrix@R=24pt@C=24pt@!0{
 	& & \bullet\\
 	&\bullet\ar@{-}[ur]^{1} &    \\
 \bullet\ar@{-}[rr]^{3}\ar@{-}[ur]^{2}& &\bullet \ar@{-}[ul]_{4}   \\  }\]
\end{small}

\bigskip

 \noindent Then the decomposition of left regular $A$-module  $_{A}A=~~\begin{matrix}1\\2\\4\\1\end{matrix}
 ~~\oplus~~\begin{matrix}2\\\begin{matrix}3\end{matrix}~~\begin{matrix}4\\1\end{matrix}
 	\\2\end{matrix}
 ~~\oplus ~~\begin{matrix}3\\\begin{matrix}2\end{matrix}~~\begin{matrix}4\end{matrix}
 	\\3\end{matrix}
 ~~\oplus~~\begin{matrix}4\\\begin{matrix}1\\2\end{matrix}~~\begin{matrix}3\end{matrix}\\4
 \end{matrix}.$
 The corresponding quiver $Q_{B}$ is the following diagram:
 $$
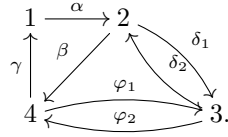
\xymatrix{
 	1  \ar[r]^{\alpha} & 2 \ar[dl]_{\beta} \ar@/^/[dr]^{\delta_{1}}   \\
 	4 \ar[u]^{\gamma} \ar@/^/[rr]^{\varphi_{1}} &   & 3. \ar@/^/[ul]_{\delta_{2}}  \ar@/^/[ll]_{\varphi_{2}}
}
 $$
 \bigskip
It is known that the  AR-quiver $\Gamma_{A}$ consists of one Euclidean component $\Gamma$, one quasi-tube $P$ of rank $3$ and  one quasi-tube $Q$ of rank $5$,  as well as infinitely many homogeneous tubes.
 
By equations (\ref{bi-perp-1}), (\ref{bi-perp-2}) and (\ref{bi-perp-3}),
\[{^{\bot}{3}^{\bot}}\bigcap{_{s}\Gamma}=\left\{\begin{matrix}1\\2\end{matrix},~~ 2,~~4,~~\begin{matrix}4\\1\end{matrix}\right\},\] 

\[{^{\bot}{3}^{\bot}}\bigcap{_{s}P}={_{s}P}\backslash \left(W_{\tiny{\begin{matrix}3\\2\end{matrix}}}\bigcup W_{\tiny{\begin{matrix}4\\3\end{matrix}}}\right)=\triangle_{\tiny{\begin{matrix}2\\4\end{matrix}}}\bigcup\triangle_{1}=\left\{\begin{matrix}2\\4\end{matrix},~~ \begin{matrix}2\\4\\1\end{matrix},~~\begin{matrix}1\\2\\4\end{matrix},~~1\right\}\]
 and 
 \[{^{\bot}{3}^{\bot}}\bigcap{_{s}Q}={_{s}Q}\backslash \left(W_{\tiny{\begin{matrix}2\\3\end{matrix}}}\bigcup W_{\tiny{\begin{matrix}3\\4\end{matrix}}}\right)=\triangle_{\tiny{\begin{matrix}4\\1\\2\end{matrix}}}=\left\{\begin{matrix}4\\1\\2\end{matrix}\right\}.\]
 
Thus 
\begin{equation} \label{orth-system-1}
\begin{aligned}
{^{\bot}{3}^{\bot}}\bigcap {_{s}\Gamma_{A}}=\left\{\begin{matrix}1\\2\end{matrix},~~ 2,~~4,~~\begin{matrix}4\\1\end{matrix},~~ \begin{matrix}2\\4\end{matrix},~~ \begin{matrix}2\\4\\1\end{matrix},~~\begin{matrix}1\\2\\4\end{matrix},~~1,~~ \begin{matrix}4\\1\\2\end{matrix},~~\begin{matrix}4\\1\\2\end{matrix}\right\}.
\end{aligned}
 \end{equation}
 
We may extend the set $R_{1}=\{3\}$ to be a maximal orthogonal system by adding some objects of ${^{\bot}{3}^{\bot}}\cap{_{s}\Gamma_{A}}$. 
By adding one object to the former orthogonal system every time, within finitely many steps, we may  construct a maximal orthogonal system containing simple module $3$. 
For example, we first add the object $\begin{matrix}1\\2\end{matrix}$ to $R_{1}$ such that $R_{2}=R_{1}\bigcup\left\{\begin{matrix}1\\2\end{matrix}\right\}$. It is clear that $R_{2}$ is an orthogonal system.
By equations (\ref{bi-perp-1}), (\ref{bi-perp-2}) and (\ref{bi-perp-3}), 
\[{^{\bot}{R_{2}}^{\bot}}\bigcap{_{s}\Gamma}=\left\{4\right\}\subseteq{^{\bot}{R_{1}}^{\bot}}\bigcap{_{s}\Gamma},\]

\[{^{\bot}{R_{2}}^{\bot}}\bigcap{_{s}P}={_{s}P}\backslash \left(W_{\tiny{\begin{matrix}3\\2\end{matrix}}}\bigcup W_{\tiny{\begin{matrix}4\\3\end{matrix}}}\bigcup W_{\tiny{\begin{matrix}1\\2\\4\end{matrix}}}\bigcup W_{1}\right)=\triangle_{\tiny{\begin{matrix}2\\4\\1\end{matrix}}}=\left\{ \begin{matrix}2\\4\\1\end{matrix}\right\}\subseteq{^{\bot}{R_{1}}^{\bot}}\bigcap{_{s}P}\]
and 
\[{^{\bot}{R_{2}}^{\bot}}\bigcap{_{s}Q}={_{s}Q}\backslash \left(W_{\tiny{\begin{matrix}2\\3\end{matrix}}}\bigcup W_{\tiny{\begin{matrix}3\\4\end{matrix}}}\bigcup W_{\tiny{\begin{matrix}4\\1\\2\end{matrix}}}\right)=\varnothing\subseteq{^{\bot}{R_{1}}^{\bot}}\bigcap{_{s}Q}.\]
Thus 
\begin{equation}
\begin{aligned}
{^{\bot}{R_{2}}^{\bot}}\bigcap {_{s}\Gamma_{A}}=\left\{4,~~\begin{matrix}2\\4\\1\end{matrix}\right\}\subseteq	{^{\bot}{R_{2}}^{\bot}}\bigcap{_{s}\Gamma_{A}}.
\end{aligned}
\end{equation}
Then add simple module $\begin{matrix}4\end{matrix}$ to $R_{2}$ such that  $R_{3}=R_{2}\cup\{\begin{matrix}4\end{matrix}\}$. Therefore 
\begin{equation}
\begin{aligned}
{^{\bot}{R_{3}}^{\bot}}\bigcap {_{s}\Gamma_{A}}=\left\{\begin{matrix}2\\4\\1\end{matrix}\right\}.
\end{aligned}
\end{equation}
It is easy to know that $R_{4}=R_{3}\bigcup\left\{\begin{matrix}2\\4\\1\end{matrix}\right\}= \left\{3,~~\begin{matrix}1\\2\end{matrix},~~ 4,~~\begin{matrix}2\\4\\1\end{matrix}\right\}$ 
is a maximal orthogonal system containing $R_{1}.$ By going through all cases, we may construct all simple-minded  systems containing  $R_{1}$ as follows.

 \[ R_{4}=\left\{3,~~\begin{matrix}1\\2\end{matrix},~~ 4,~~\begin{matrix}2\\4\\1\end{matrix}\right\},
 \left\{3,~~\begin{matrix}2\\4\\1\end{matrix},~~\begin{matrix}1\\2\\4\end{matrix},~~ \begin{matrix}4\\1\\2\end{matrix}\right\},
 \left\{3,~~\begin{matrix}2\\4\end{matrix},~~ 1,~~\begin{matrix}4\\1\\2\end{matrix}\right\},
 \left\{3,~~2,~~ 4,~~1\right\}, \left\{3,~~2,~~\begin{matrix}4\\1\end{matrix},~~ \begin{matrix}1\\2\\4\end{matrix}\right\}.
 \]
 By Corollary \ref{sms-BGA}, They are all simple-minded systems containing simple module $3$ in $A$-$\stmod$. In this way, we may list all simple-minded systems in $A$-$\stmod$.
  
 We draw parts of the stable AR-components $_{s}\Gamma$, $_{s}P$ and $_{s}Q$  respectively as in the figure 1, 2 and 3. Note that the blue part in the diagram is contained in the stable bi-perpendicular category of simple module $3$.
\begin{figure}
\vspace{-2.5cm}
\begin{small}
\[\xymatrix@dr@R=22pt@C=22pt@!0{
&&  &&\scriptstyle \cdots\ar[rr] &&\scriptstyle \new{\tiny{\begin{matrix}4\\1\\2\end{matrix}}} && \\	
&&& \\
&& \scriptstyle \cdots \ar[rr]  &&\scriptstyle \tiny{\begin{matrix}2\\3\end{matrix}} \ar[uu]  \\
&&& \\
\scriptstyle \cdots\ar[rr] && \scriptstyle \tiny{\begin{matrix}3\\4\end{matrix}}\ar[uu]\\	
&& \\
\new{\tiny{\begin{matrix}4\\1\\2\end{matrix}}} \ar[uu] \\}\] 
\end{small}
\vspace{-1cm}
\caption{Stable quasi-tube $_{s}P$ of rank 3}
\end{figure}

\newpage
\begin{figure}
\vspace{-4.5cm}
\begin{small}
\[\xymatrix@dr@R=22pt@C=22pt@!0{
&&&&\scriptstyle &&\scriptstyle &&\scriptstyle \cdots\ar[rr]&&\scriptstyle\ar[rr]\new{\tiny{\begin{matrix}2\\4\end{matrix}}}&& \new{\tiny{\begin{matrix}1\\2\\4\end{matrix}}}\\
&&\\
&&\scriptstyle & &\scriptstyle  && \scriptstyle \cdots\ar[rr]&& \scriptstyle \tiny{\begin{matrix} 4\begin{matrix}\ \end{matrix}~~\begin{matrix}2 \end{matrix}\\\begin{matrix}\ \end{matrix}\begin{matrix}\ \end{matrix}\begin{matrix}\ \end{matrix}\begin{matrix}\ \end{matrix}~~\begin{matrix}3\end{matrix}~~\begin{matrix}4\\1\end{matrix}\end{matrix}} \ar[uu] \ar[rr]  &&\scriptstyle \new{\tiny{\begin{matrix}2\\4\\1\end{matrix}}} \ar[uu]&&\\	
&& \\	
\scriptstyle&&\scriptstyle  &&\scriptstyle \cdots \ar[rr] &&\scriptstyle \tiny{\begin{matrix}4\\\begin{matrix}1\end{matrix}~~\begin{matrix}3\end{matrix}\end{matrix}}\ar[uu]\ar[rr]&&\scriptstyle \tiny{\begin{matrix}4\\3\end{matrix}}\ar[uu]\\	
&& \\	
\scriptstyle  &&\scriptstyle \cdots\ar[rr]  &&\scriptstyle \tiny{\begin{matrix}\begin{matrix}3\end{matrix}~~\begin{matrix}1\end{matrix}\\2\end{matrix}}\ar[uu]\ar[rr] &&\scriptstyle \new{1}\ar[uu]  && \\	
&&& \\
\scriptstyle \cdots\ar[rr] && \scriptstyle \tiny{\begin{matrix} \begin{matrix}\ \end{matrix}~~\begin{matrix}\ \end{matrix}~~\begin{matrix}3\end{matrix}~~\begin{matrix}1\\2\end{matrix}\\\begin{matrix}2\end{matrix}~~4 \end{matrix}} \ar[uu] \ar[rr]  &&\scriptstyle \tiny{\begin{matrix}3\\2\end{matrix}} \ar[uu]  \\
&&& \\
\scriptstyle \new{\tiny{\begin{matrix}2\\4\end{matrix}}} \ar[uu]\ar[rr] && \scriptstyle \new{\tiny{\begin{matrix}1\\2\\4\end{matrix}}}\ar[uu]\\	
&& \\
\new{\tiny{\begin{matrix}2\\4\\1\end{matrix}}} \ar[uu] \\}\] 
\end{small}
\vspace{-1cm}
\caption{Stable quasi-tube $_{s}Q$ of rank 5}
\end{figure}

\begin{figure}
\vspace{-1cm}
\begin{equation*}\label{AR-gamma-0}
\xymatrix@dr@R=25pt@C=25pt@!0{
&&\scriptstyle 3\ar[rr] &&\scriptstyle \tiny{\begin{matrix}4\\ \begin{matrix}3\end{matrix}~~\begin{matrix}1\\2\end{matrix} \end{matrix}} \ar[rr]& &\scriptstyle
\tiny{\begin{matrix}  \begin{matrix}4\end{matrix}~~\begin{matrix}\ \end{matrix}\\\begin{matrix}3\end{matrix}~~\begin{matrix}1\end{matrix}~~\begin{matrix}3\end{matrix}\\\begin{matrix}\ \end{matrix}~~\begin{matrix}\ \end{matrix}~~\begin{matrix}\ \end{matrix}~~\begin{matrix}2 \end{matrix}~~\begin{matrix}4\end{matrix}\end{matrix}}
\ar[rr] && \scriptstyle\tiny{\begin{matrix}  \begin{matrix}4\end{matrix}~~\begin{matrix}\ \end{matrix}~~\begin{matrix}\ \end{matrix}~~\begin{matrix}\ \end{matrix}\\\begin{matrix}3\end{matrix}~~\begin{matrix}1\end{matrix}~~\begin{matrix}3\end{matrix}~~\begin{matrix}2\end{matrix}\\\begin{matrix}\ \end{matrix}~~\begin{matrix}\ \end{matrix}~~\begin{matrix}\ \end{matrix}~~\begin{matrix}2 \end{matrix}~~\begin{matrix}4\end{matrix}~~\begin{matrix}3\end{matrix}\end{matrix}}\\	
&& \\	
&&\scriptstyle \tiny{\begin{matrix} \begin{matrix}3\end{matrix}~~\begin{matrix}1\\2\end{matrix}\\4 \end{matrix}} \ar[uu] \ar[rr] &&\scriptstyle\new{\tiny{\begin{matrix}1\\2\end{matrix}}}\ar[rr]\ar[uu]&& 	\scriptstyle\tiny{\begin{matrix}\begin{matrix}1\end{matrix}~~\begin{matrix}3\end{matrix}\\\begin{matrix}\ \end{matrix}~~\begin{matrix}2 \end{matrix}~~\begin{matrix}4\end{matrix}\end{matrix}} \ar[rr]\ar[uu]&&\scriptstyle \tiny{\begin{matrix}  \\\begin{matrix}1\end{matrix}~~\begin{matrix}3\end{matrix}~~\begin{matrix}2\end{matrix}\\\begin{matrix}\ \end{matrix}~~\begin{matrix}2 \end{matrix}~~\begin{matrix}4\end{matrix}~~\begin{matrix}3\end{matrix}\end{matrix}} \ar[uu] \\	
&& \\	
&&\scriptstyle \tiny{\begin{matrix} \begin{matrix}3\end{matrix}~~\begin{matrix}2\end{matrix}\\4 \end{matrix}} \ar[uu]\ar[rr] &&\scriptstyle \new{2}\ar[rr]\ar[uu] &&\scriptstyle \old{\Omega^{-1}(3)=\tiny{\begin{matrix}\begin{matrix}3\end{matrix}\\\begin{matrix}2 \end{matrix}~~\begin{matrix}4\end{matrix}\end{matrix}}}\ar[rr]\ar[uu] && \scriptstyle\tiny{\begin{matrix}  \\\begin{matrix}\ \end{matrix}~~\begin{matrix}3\end{matrix}~~\begin{matrix}2\end{matrix}\\\begin{matrix}\ \end{matrix}~~\begin{matrix}2 \end{matrix}~~\begin{matrix}4\end{matrix}~~\begin{matrix}3\end{matrix}\end{matrix}}\ar[uu] \\	
&&	& \\
&&\scriptstyle \tiny{\begin{matrix} \begin{matrix}3\end{matrix}~~\begin{matrix}2\end{matrix}~~\begin{matrix}4\end{matrix}\\ 4 ~~\begin{matrix}\ \end{matrix}~~\begin{matrix}3\end{matrix}\end{matrix}} \ar[rr] \ar[uu]&& \scriptstyle\old{\Omega(3)=\tiny{\begin{matrix} \begin{matrix}2\end{matrix}~~\begin{matrix}4\end{matrix}\\ \begin{matrix}3\end{matrix}\end{matrix}} }\ar[rr]\ar[uu] &&\scriptstyle \new{4}\ar[rr] \ar[uu] && \scriptstyle\tiny{\begin{matrix} \\ \begin{matrix}2\end{matrix}\\\begin{matrix}4\end{matrix}~~\begin{matrix}3\end{matrix}\end{matrix}}\ar[uu] \\
&&& \\
&&\scriptstyle\tiny{\begin{matrix} \begin{matrix}3\end{matrix}~~\begin{matrix}2\end{matrix}~~\begin{matrix}4\end{matrix}~~\begin{matrix}\ \end{matrix}\\ 4 ~~\begin{matrix}\ \end{matrix}~~\begin{matrix}3\end{matrix}~~\begin{matrix}1\end{matrix}\end{matrix}} \ar[uu] \ar[rr] && \scriptstyle\tiny{\begin{matrix} \begin{matrix}2\end{matrix}~~\begin{matrix}4\end{matrix}~~\begin{matrix}\ \end{matrix}\\ \begin{matrix}\ \end{matrix}~~\begin{matrix}3\end{matrix}~~\begin{matrix}1\end{matrix}\end{matrix}}  \ar[rr]\ar[uu] && \scriptstyle\new{\tiny{\begin{matrix}4\\1\end{matrix}}}\ar[rr]\ar[uu]
&&\scriptstyle\tiny{\begin{matrix}  \\\begin{matrix}2\end{matrix}\\ \begin{matrix}4\\1\end{matrix}~~\begin{matrix}3\end{matrix}\end{matrix}} \ar[uu] \\ 
&&\\
&&\scriptstyle\tiny{\begin{matrix}\begin{matrix}3\end{matrix}~~\begin{matrix}2\end{matrix}~~\begin{matrix}4\end{matrix}~~\begin{matrix}\ \end{matrix}~~\begin{matrix}\ \end{matrix}~~\begin{matrix}\ \end{matrix}\\ 4 ~~\begin{matrix}\ \end{matrix}~~\begin{matrix}3\end{matrix}~~\begin{matrix}1\end{matrix}~~\begin{matrix} 3\end{matrix}\\~~\begin{matrix}\ \end{matrix}~~\begin{matrix}\ \end{matrix}~~\begin{matrix}\ \end{matrix}~~\begin{matrix}2\end{matrix}\end{matrix}}  \ar[uu] \ar[rr]&&\scriptstyle\tiny{\begin{matrix} \begin{matrix}2\end{matrix}~~\begin{matrix}4\end{matrix}~~\begin{matrix}\ \end{matrix}~~\begin{matrix}\ \end{matrix}~~\begin{matrix}\ \end{matrix}\\ \begin{matrix}3\end{matrix}~~\begin{matrix}1\end{matrix}~~\begin{matrix} 3\end{matrix}\\~~\begin{matrix}\ \end{matrix}~~\begin{matrix}2\end{matrix}\end{matrix}} \ar[rr]\ar[uu]&&\scriptstyle~~\tiny{\begin{matrix}\begin{matrix}3\end{matrix}~~\begin{matrix}4\\1\end{matrix}\\2\end{matrix}}\ar[rr]\ar[uu]&&\scriptstyle3\ar[uu] }
\end{equation*}
\caption{Stable Euclidean component $_{s}\Gamma$}
\end{figure}
\end{Ex}

\bigskip

\subsection{The case that $\overline{\varphi}$ inverses the orientation of the irreducible maps in the Euclidean components.}
By  Proposition \ref{repetitive-stable generalized}, every connected component of  $\{P_{n}\}_{n\in\mathbb{Z}}$ and  $\{Q_{n}\}_{n\in\mathbb{Z}}$ in the AR-quiver $\Gamma_{\widehat{B}}$ is  stable generalized standard.  By Lemma \ref{quasi-tube-varphi}, $\overline{\varphi}(_{s}P_{n})={_{s}Q_{n+1}}$ and $\overline{\varphi}(_{s}Q_{n})={_{s}P_{n+1}}$ in $\widehat{B}$-$\stmod$. It is known that 
$\Omega^{-1}_{\widehat{B}}(_{s}P_{n})={_{s}P_{n+1}}$ and $\Omega_{\widehat{B}}^{-1}(_{s}Q_{n})={_{s}Q_{n+1}}$ in $\widehat{B}$-$\stmod$. 
By covering theory, 
 Quasi-tubes $P$ of rank $p$ and $Q$ of rank $q$ of $\Gamma^{s}_{A}$ satisfy  conditions $\Omega(_{s}P)={_{s}Q}$ and $\Omega(_{s}Q)={_{s}P}$, and both $P$ and $Q$ are stable generalized standard components. Note that $p=q=n$ is the number of non-projective simple $A$-modules.
We take an example as follows. 
\begin{Ex}\label{1-domestic-2}
Consider Brauer graph algebra $C$ with Brauer graph $D$ given by
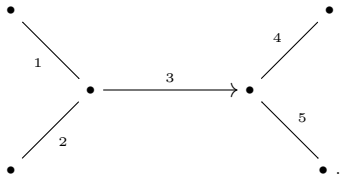
\begin{center}
\begin{tiny}	
\[\xymatrix@R=30pt@C=30pt@!0{\bullet\ar@{-}[dr]_{1}& & &&\bullet\\
& \bullet\ar[rr]^{3}&&  \bullet\ar@{-}[dr]^{5}\ar@{-}[ur]^{4}&  \\
\bullet\ar@{-}[ur]_{2}& & && \bullet \ .\\  }\]
\end{tiny}
\end{center}
	
\noindent Then the decomposition of left regular $C$-module  $_{C}C=~~\begin{matrix}1\\\begin{matrix}1\end{matrix}~~\begin{matrix}2\\3\end{matrix}\\1\end{matrix}
~~\oplus~~\begin{matrix}2\\3\\1\\2\end{matrix}
~~\oplus~~\begin{matrix}3\\\begin{matrix}1\\2\end{matrix}~~\begin{matrix}4\\5\end{matrix}
\\3\end{matrix}
~~\oplus ~~\begin{matrix}4\\\begin{matrix}5\\3\end{matrix}~~\begin{matrix}4\end{matrix}
\\4\end{matrix}
~~\oplus~~\begin{matrix}5\\3\\4\\5\end{matrix}\ .$
The corresponding quiver $Q_{D}$ is the following diagram:
$$
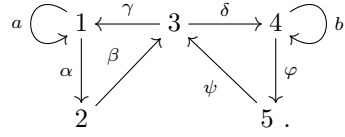
\xymatrix{
& 1\ar@(ul,dl)[]_{a}   \ar[d]_{\alpha} & 3 \ar[l]_{\gamma}\ar[r]^{\delta} & 4 \ar@(ur,dr)[]^{b} \ar[d]^{\varphi}\\
&2 \ar[ur]^{\beta}&   & 5\ .\ar[ul]^{\psi} }
$$
 \bigskip
It is known that the stable AR-quiver $\Gamma_{C}$ consists of one Euclidean component $\Gamma$, two quasi-tubes $P$  and $Q$ of rank $5$,  as well as infinitely many homogeneous tubes.	

We shall list all simple-minded systems containing  orthogonal system $S_{1}=\left\{ 1,~~\begin{matrix}5\\3\end{matrix}\right\}.$
By equations (\ref{bi-perp-1}), (\ref{bi-perp-2}) and (\ref{bi-perp-3}), 
\[{^{\bot}{S_{1}}^{\bot}}\bigcap{_{s}\Gamma}=\left\{\begin{matrix}3\\ \begin{matrix}1\\2\end{matrix}~~\begin{matrix}4\\5\end{matrix} \end{matrix},~~4,~~\begin{matrix}2\\3\end{matrix}\right\},\]

\[{^{\bot}{S_{1}}^{\bot}}\bigcap{_{s}P}={_{s}P}\backslash \left(W_{\tiny{\begin{matrix}2\\3\\1\end{matrix}}}\bigcup W_{\tiny{\begin{matrix}1\\2\\3\end{matrix}}}\right)=\triangle_{\tiny{ \begin{matrix}4\\4\end{matrix}}}=\left\{\begin{matrix}4\\4\end{matrix}\right\}\]
and 
\[{^{\bot}{S_{1}}^{\bot}}\bigcap {_{s}Q}={_{s}Q}\backslash \left(W_{\tiny{\begin{matrix}1\\1\end{matrix}}}\bigcup W_{\tiny{\begin{matrix}5\\3\\4\end{matrix}}}\bigcup W_{\tiny{\begin{matrix}4\\5\\3\end{matrix}}}\right)=\triangle_{\tiny{\begin{matrix}3\\4\\5\end{matrix}}}\bigcup\triangle_{\tiny{2}}=\left\{\begin{matrix}3\\4\\5\end{matrix},~~2\right\}.\]
Thus 
\begin{equation}\label{bi=pen-1}
\begin{aligned}
{^{\bot}{S_{1}}^{\bot}}\bigcap{ _{s}\Gamma_{C}}=\left\{\begin{matrix} \begin{matrix}3\end{matrix}\\ \begin{matrix}1\\2\end{matrix}~~\begin{matrix}4\\5\end{matrix} \end{matrix},~~4,~~\begin{matrix}2\\3\end{matrix},~~ \begin{matrix}4\\4\end{matrix},~~\begin{matrix}3\\4\\5\end{matrix},~~2\right\}.
\end{aligned}
\end{equation}
Our aim is to extend $S_{1}$ to be a maximal orthogonal system. 
By equation (\ref{bi=pen-1}), We may add an object $X$ of  ${^{\bot}{S_{1}}^{\bot}}\cap {_{s}\Gamma_{A}}$ to the set $S_{1}$ such that $S_{2}=S_{1}\cup\{X\}. $ By equations (\ref{bi-perp-1}), (\ref{bi-perp-2}) and (\ref{bi-perp-3}), we may determine ${^{\bot}{S_{2}}^{\bot}}\cap {_{s}\Gamma_{A}}$.  Proceeding this process, within finitely many steps,  we  may construct a maximal orthogonal systems containing  $S_{1}$. 
For example, add the object \, $\begin{matrix} \begin{matrix}3\end{matrix}\\ \begin{matrix}1\\2\end{matrix}~~\begin{matrix}4\\5\end{matrix} \end{matrix}$ to $S_{1}$ such that \[S_{2}=S_{1}\cup\left\{\begin{matrix} \begin{matrix}3\end{matrix}\\ \begin{matrix}1\\2\end{matrix}~~\begin{matrix}4\\5\end{matrix} \end{matrix}\right\}=\left\{1,~~\begin{matrix}5\\3\end{matrix},~~\begin{matrix} \begin{matrix}3\end{matrix}\\ \begin{matrix}1\\2\end{matrix}~~\begin{matrix}4\\5\end{matrix} \end{matrix}\right\}.\]
By equations (\ref{bi-perp-1}), (\ref{bi-perp-2}) and (\ref{bi-perp-3}), 
\[{^{\bot}{S_{2}}^{\bot}}\bigcap{_{s}\Gamma}=\left\{4,~~\begin{matrix}2\\3\end{matrix}\right\}\subseteq{^{\bot}{S_{1}}^{\bot}}\bigcap{_{s}\Gamma},\]

\[{^{\bot}{S_{2}}^{\bot}}\bigcap{_{s}P}={_{s}P}\backslash \left(W_{\tiny{\begin{matrix}2\\3\\1\end{matrix}}}\bigcup W_{\tiny{\begin{matrix}1\\2\\3\end{matrix}}}\right)=\triangle_{\tiny{ \begin{matrix}4\\4\end{matrix}}}=\left\{\begin{matrix}4\\4\end{matrix}\right\}={^{\bot}{S_{1}}^{\bot}}\bigcap{_{s}P}.\]
and 
\[{^{\bot}{S_{2}}^{\bot}}\bigcap{ _{s}Q}={_{s}Q}\backslash \left(W_{\tiny{\begin{matrix}1\\1\end{matrix}}}\bigcup W_{\tiny{\begin{matrix}5\\3\\4\end{matrix}}}\bigcup W_{\tiny{\begin{matrix}4\\5\\3\end{matrix}}}\bigcup W_{\tiny{2}}\bigcup W_{\tiny{\begin{matrix}3\\4\\5\end{matrix}}}\right)=\varnothing\subseteq{^{\bot}{S_{1}}^{\bot}}\bigcap{_{s}Q}.\]
Thus 
\begin{equation}
\begin{aligned}
{^{\bot}{S_{2}}^{\bot}}\bigcap {_{s}\Gamma_{C}}=\left\{4,~~\begin{matrix}2\\3\end{matrix},~~\begin{matrix}4\\4\end{matrix}\right\}\subseteq	{^{\bot}{S_{2}}^{\bot}}\bigcap {_{s}\Gamma_{C}}.
\end{aligned}
\end{equation}
Add the object \, $\begin{matrix}2\\3\end{matrix}$ to $S_{2}$ such that $S_{3}=S_{2}\cup\left\{\begin{matrix}2\\3\end{matrix}\right\}$. Then 

\begin{equation}
\begin{aligned}
{^{\bot}{S_{3}}^{\bot}}\bigcap {_{s}\Gamma_{A}}=\left\{4,~~\begin{matrix}4\\4\end{matrix}\right\}.
\end{aligned}
\end{equation}
It is easy to know that $S_{4}=S_{3}\cup\left\{4\right\}=\left\{1,~~\begin{matrix}5\\3\end{matrix},~~\begin{matrix}3\\ \begin{matrix}1\\2\end{matrix}~~\begin{matrix}4\\5\end{matrix} \end{matrix},~~\begin{matrix}2\\3\end{matrix},~~4\right\}$ 
and $S'_{4}=S_{3}\cup\left\{\begin{matrix}4\\4\end{matrix}\right\}=\left\{1,~~\begin{matrix}5\\3\end{matrix},~~\begin{matrix}3\\ \begin{matrix}1\\2\end{matrix}~~\begin{matrix}4\\5\end{matrix} \end{matrix},~~\begin{matrix}2\\3\end{matrix},~~\begin{matrix}4\\4\end{matrix}\right\}$ 
are maximal orthogonal systems containing the set $S_{1}.$ 
By the same method, we may construct all maximal orthogonal systems containing the set $S_{1}$ and we list them as follows. 
\[S_{4}=\left\{1,~~\begin{matrix}5\\3\end{matrix},~~\begin{matrix}3\\ \begin{matrix}1\\2\end{matrix}~~\begin{matrix}4\\5\end{matrix} \end{matrix},~~4,~~\begin{matrix}2\\3\end{matrix}\right\},
S'_{4}=\left\{1,~~\begin{matrix}5\\3\end{matrix},~~\begin{matrix}3\\ \begin{matrix}1\\2\end{matrix}~~\begin{matrix}4\\5\end{matrix} \end{matrix},~~\begin{matrix}2\\3\end{matrix},~~\begin{matrix}4\\4\end{matrix}\right\},
\left\{1,~~\begin{matrix}5\\3\end{matrix},~~4,~~\begin{matrix}3\\4\\5\end{matrix},~~2\right\},
\left\{1,~~\begin{matrix}5\\3\end{matrix},~~\begin{matrix}4\\4\end{matrix},~~\begin{matrix}3\\4\\5\end{matrix},~~2\right\}.
\]
By Corollary \ref{sms-BGA}, They are all simple-minded systems containing orthogonal system $S_{1}$ in $C$-$\stmod$. In this way, we may list all  simple-minded systems in $C$-$\stmod$.

We draw parts of the stable AR-components $_{s}\Gamma$, $_{s}P$ and $_{s}Q$ as in the figure 4, 5 and 6, respectively. 
\begin{figure}\vspace{-5cm}
\begin{small}
\[\xymatrix@dr@R=27pt@C=27pt@!0{
&& &&&&\scriptstyle &&\scriptstyle\cdots\ar[rr]&&\scriptstyle\tiny{\begin{matrix} \begin{matrix}3\end{matrix}~~\begin{matrix}1\end{matrix}~~\begin{matrix}\ \end{matrix}\\\begin{matrix}\ \end{matrix}\begin{matrix}1\end{matrix}~~\begin{matrix}2\end{matrix}\\ \begin{matrix}\ \end{matrix}~~\begin{matrix}\ \end{matrix}~~\begin{matrix}\ \end{matrix}~~\begin{matrix}3\end{matrix} \end{matrix}} \ar[rr]&& \scriptstyle\cdots&& &&\\
&&\\
&&&&\scriptstyle &&\scriptstyle\cdots \ar[rr] &&\scriptstyle\tiny{\begin{matrix} \begin{matrix}4\end{matrix}~~\begin{matrix}3\end{matrix}~~\begin{matrix}\ \end{matrix}\\\begin{matrix}\ \end{matrix}\begin{matrix}4\end{matrix}~~\begin{matrix}1\end{matrix}\end{matrix}} \ar[rr]\ar[uu] \ar[rr]&&\scriptstyle\ar[rr]\ar[uu]\tiny{\begin{matrix}3\\1\end{matrix}}&&\scriptstyle\ar[uu] \tiny{\begin{matrix}2\\3\\1\end{matrix}}\\
&&\\
&&\scriptstyle& &\scriptstyle\cdots\ar[rr] && \scriptstyle\tiny{\begin{matrix} ~~\begin{matrix}\ \end{matrix}~~\begin{matrix}4\end{matrix}~~\begin{matrix}3 \end{matrix}~~\begin{matrix}\ \end{matrix}~~\begin{matrix}\ \end{matrix}\\\begin{matrix}5 \end{matrix}~~\begin{matrix}4 \end{matrix}~~\begin{matrix}1 \end{matrix}~~\begin{matrix}\ \end{matrix}\\ \begin{matrix}\ \end{matrix}~~\begin{matrix}\ \end{matrix}~~\begin{matrix}\ \end{matrix}~~\begin{matrix}2\end{matrix} \end{matrix}}\ar[rr]\ar[uu]&& \scriptstyle \tiny{\begin{matrix} \begin{matrix}4\end{matrix}~~\begin{matrix}3\end{matrix}~~\begin{matrix}\ \end{matrix}\\\begin{matrix}\ \end{matrix}\begin{matrix}4\end{matrix}~~\begin{matrix}1\end{matrix}\\ \begin{matrix}\ \end{matrix}~~\begin{matrix}\ \end{matrix}~~\begin{matrix}\ \end{matrix}~~\begin{matrix}2\end{matrix} \end{matrix}}  \ar[uu] \ar[rr]  &&\scriptstyle \tiny{\begin{matrix}3\\1\\2\end{matrix}}\ar[uu]&&\\	
&& \\	
\scriptstyle &&\scriptstyle \cdots\ar[rr] &&\scriptstyle \tiny{\begin{matrix}\begin{matrix}\ \end{matrix}~~\begin{matrix}1 \end{matrix}~~\begin{matrix}\ \end{matrix}~~\begin{matrix}\ \end{matrix}\begin{matrix}4\end{matrix}\\~~\begin{matrix}\ \end{matrix}~~\begin{matrix}\ \end{matrix}~~\begin{matrix}2\end{matrix}~~~~\begin{matrix}\ \end{matrix}\begin{matrix}5\end{matrix}~~\begin{matrix}4\end{matrix}\\\begin{matrix}\ \end{matrix}~~\begin{matrix}3\end{matrix}\end{matrix}}\ar[uu] \ar[rr] &&\scriptstyle \tiny{\begin{matrix}4\\\begin{matrix}4\end{matrix}~~\begin{matrix}5\end{matrix}\end{matrix}}\ar[uu]\ar[rr]&&\scriptstyle \new{\tiny{\begin{matrix}4\\4\end{matrix}}}\ar[uu]\\	
&& \\	
\scriptstyle\cdots \ar[rr]  &&\scriptstyle \tiny{\begin{matrix} \begin{matrix}2\end{matrix}~~\begin{matrix}\ \end{matrix}~~\begin{matrix}\ \end{matrix}~~\begin{matrix}\ \end{matrix}~~\\\begin{matrix}3\end{matrix}~~\begin{matrix}1\end{matrix}~~\begin{matrix}\ \end{matrix}\\\begin{matrix}\ \end{matrix}~~\begin{matrix}\ \end{matrix}~~\begin{matrix}1\end{matrix}~~\begin{matrix}2\end{matrix}~~\begin{matrix}5\end{matrix}\\ \begin{matrix}\ \end{matrix}~~\begin{matrix}\ \end{matrix}~~\begin{matrix}\ \end{matrix}~~\begin{matrix}3\end{matrix} \end{matrix}} \ar[rr] \ar[uu] &&\scriptstyle \tiny{\begin{matrix}\begin{matrix}1\\2\end{matrix}~~\begin{matrix}5\end{matrix}\\3\end{matrix}}\ar[uu]\ar[rr] &&\scriptstyle5\ar[uu]  && \\	
&&& \\
\scriptstyle \tiny{\begin{matrix} \begin{matrix}3\end{matrix}~~\begin{matrix}1\end{matrix}~~\begin{matrix}\ \end{matrix}\\\begin{matrix}\ \end{matrix}\begin{matrix}1\end{matrix}~~\begin{matrix}2\end{matrix}\\ \begin{matrix}\ \end{matrix}~~\begin{matrix}\ \end{matrix}~~\begin{matrix}\ \end{matrix}~~\begin{matrix}3\end{matrix} \end{matrix}} \ar[uu]\ar[rr] && \scriptstyle \tiny{\begin{matrix} \begin{matrix}2\end{matrix}~~\begin{matrix}\ \end{matrix}~~\begin{matrix}\ \end{matrix}~~\begin{matrix}\ \end{matrix}~~\\\begin{matrix}3\end{matrix}~~\begin{matrix}1\end{matrix}~~\begin{matrix}\ \end{matrix}\\\begin{matrix}\ \end{matrix}\begin{matrix}1\end{matrix}~~\begin{matrix}2\end{matrix}\\ \begin{matrix}\ \end{matrix}~~\begin{matrix}\ \end{matrix}~~\begin{matrix}\ \end{matrix}~~\begin{matrix}3\end{matrix} \end{matrix}} \ar[uu] \ar[rr]  &&\scriptstyle \tiny{\begin{matrix}1\\2\\3\end{matrix}} \ar[uu]  \\
&&& \\
\scriptstyle \tiny{\begin{matrix}3\\1\end{matrix}} \ar[uu]\ar[rr] && \scriptstyle \tiny{\begin{matrix}2\\3\\1\end{matrix}}\ar[uu]\\	
&& \\
\tiny{\begin{matrix}3\\1\\2\end{matrix}} \ar[uu] \\}\] 
\end{small}
\vspace{-0.5cm}
\caption{Stable quasi-tube $_{s}P$ of rank $5$}
\end{figure}

\begin{figure}\vspace{-5cm}
\begin{small}
\[\xymatrix@dr@R=27pt@C=24pt@!0{
&& &&&&\scriptstyle &&\scriptstyle\cdots \ar[rr]&&\scriptstyle\tiny{\begin{matrix} ~~\begin{matrix}\ \end{matrix}~~\begin{matrix}1\end{matrix}~~\begin{matrix}3 \end{matrix}~~\begin{matrix}\ \end{matrix}~~\begin{matrix}\ \end{matrix}\\\begin{matrix}2 \end{matrix}~~\begin{matrix}1 \end{matrix}~~\begin{matrix}4 \end{matrix}~~\begin{matrix}\ \end{matrix}\\ \begin{matrix}\ \end{matrix}~~\begin{matrix}\ \end{matrix}~~\begin{matrix}\ \end{matrix}~~\begin{matrix}5\end{matrix} \end{matrix}} \ar[rr]&& \scriptstyle\cdots&& &&\\
&&\\
&&&&\scriptstyle &&\scriptstyle\cdots \ar[rr] &&\scriptstyle\tiny{\begin{matrix}\begin{matrix}\ \end{matrix}~~\begin{matrix}4 \end{matrix}~~\begin{matrix}\ \end{matrix}~~\begin{matrix}\ \end{matrix}\begin{matrix}1\end{matrix}\\~~\begin{matrix}\ \end{matrix}~~\begin{matrix}\ \end{matrix}~~\begin{matrix}5\end{matrix}~~~~\begin{matrix}\ \end{matrix}\begin{matrix}2\end{matrix}~~\begin{matrix}1\end{matrix}\\\begin{matrix}\ \end{matrix}~~\begin{matrix}3\end{matrix}\end{matrix}}\ar[rr]\ar[uu] \ar[rr]&&\scriptstyle\ar[rr]\ar[uu]\tiny{\begin{matrix}1\\\begin{matrix}1\end{matrix}~~\begin{matrix}2\end{matrix}\end{matrix}}&&\scriptstyle\ar[uu] \tiny{\begin{matrix}1\\1\end{matrix}}\\
&&\\
&&\scriptstyle & &\scriptstyle \cdots\ar[rr] && \scriptstyle \tiny{\begin{matrix} \begin{matrix}5\end{matrix}~~\begin{matrix}\ \end{matrix}~~\begin{matrix}\ \end{matrix}~~\begin{matrix}\ \end{matrix}~~\\\begin{matrix}3\end{matrix}~~\begin{matrix}4\end{matrix}~~\begin{matrix}\ \end{matrix}\\\begin{matrix}\ \end{matrix}~~\begin{matrix}\ \end{matrix}~~\begin{matrix}4\end{matrix}~~\begin{matrix}5\end{matrix}~~\begin{matrix}2\end{matrix}\\ \begin{matrix}\ \end{matrix}~~\begin{matrix}\ \end{matrix}~~\begin{matrix}\ \end{matrix}~~\begin{matrix}3\end{matrix} \end{matrix}}  \ar[rr]\ar[uu]&& \scriptstyle \tiny{\begin{matrix}\begin{matrix}2\end{matrix}~~\begin{matrix}4\\5\end{matrix}\\3\end{matrix}}\ar[uu] \ar[rr]  &&\scriptstyle \new{\tiny{\begin{matrix}2\end{matrix}}} \ar[uu]&&\\	
&& \\	
\scriptstyle &&\scriptstyle\cdots \ar[rr] &&\scriptstyle \tiny{\begin{matrix} \begin{matrix}3\end{matrix}~~\begin{matrix}4\end{matrix}~~\begin{matrix}\ \end{matrix}\\\begin{matrix}\ \end{matrix}\begin{matrix}4\end{matrix}~~\begin{matrix}5\end{matrix}\\ \begin{matrix}\ \end{matrix}~~\begin{matrix}\ \end{matrix}~~\begin{matrix}\ \end{matrix}~~\begin{matrix}3\end{matrix} \end{matrix}} \ar[uu] \ar[rr] &&\scriptstyle \tiny{\begin{matrix} \begin{matrix}5\end{matrix}~~\begin{matrix}\ \end{matrix}~~\begin{matrix}\ \end{matrix}~~\begin{matrix}\ \end{matrix}~~\\\begin{matrix}3\end{matrix}~~\begin{matrix}4\end{matrix}~~\begin{matrix}\ \end{matrix}\\\begin{matrix}\ \end{matrix}\begin{matrix}4\end{matrix}~~\begin{matrix}5\end{matrix}\\ \begin{matrix}\ \end{matrix}~~\begin{matrix}\ \end{matrix}~~\begin{matrix}\ \end{matrix}~~\begin{matrix}3\end{matrix} \end{matrix}} \ar[uu]\ar[rr]&&\scriptstyle \tiny{\begin{matrix}4\\5\\3\end{matrix}}\ar[uu]\\	
&& \\	
\scriptstyle \cdots\ar[rr]  &&\scriptstyle \tiny{\begin{matrix} \begin{matrix}1\end{matrix}~~\begin{matrix}1\end{matrix}~~\begin{matrix}\ \end{matrix}\\\begin{matrix}\ \end{matrix}\begin{matrix}3\end{matrix}~~\begin{matrix}4\end{matrix}\end{matrix}}  \ar[rr] \ar[uu] &&\scriptstyle \tiny{\begin{matrix}3\\4\end{matrix}}\ar[uu]\ar[rr] &&\scriptstyle \tiny{\begin{matrix}5\\3\\4\end{matrix}}\ar[uu]  && \\	
&&& \\
\scriptstyle \scriptstyle \tiny{\begin{matrix} ~~\begin{matrix}\ \end{matrix}~~\begin{matrix}1\end{matrix}~~\begin{matrix}3 \end{matrix}~~\begin{matrix}\ \end{matrix}~~\begin{matrix}\ \end{matrix}\\\begin{matrix}2 \end{matrix}~~\begin{matrix}1 \end{matrix}~~\begin{matrix}4 \end{matrix}~~\begin{matrix}\ \end{matrix}\\ \begin{matrix}\ \end{matrix}~~\begin{matrix}\ \end{matrix}~~\begin{matrix}\ \end{matrix}~~\begin{matrix}5\end{matrix} \end{matrix}} \ar[uu]\ar[rr] && \scriptstyle \tiny{\begin{matrix} \begin{matrix}1\end{matrix}~~\begin{matrix}3\end{matrix}~~\begin{matrix}\ \end{matrix}\\\begin{matrix}\ \end{matrix}\begin{matrix}1\end{matrix}~~\begin{matrix}4\end{matrix}\\ \begin{matrix}\ \end{matrix}~~\begin{matrix}\ \end{matrix}~~\begin{matrix}\ \end{matrix}~~\begin{matrix}5\end{matrix} \end{matrix}} \ar[uu] \ar[rr]  &&\scriptstyle \new{\tiny{\begin{matrix}3\\4\\5\end{matrix}}} \ar[uu]  \\
&&& \\
\scriptstyle \tiny{\begin{matrix}1\\\begin{matrix}1\end{matrix}~~\begin{matrix}2\end{matrix}\end{matrix}} \ar[uu]\ar[rr] && \scriptstyle\tiny{\begin{matrix}1\\1\end{matrix}}\ar[uu]\\	
&& \\
\new{\tiny{\begin{matrix}2\end{matrix}}} \ar[uu] }\] 
\end{small}
\vspace{-0.5cm}
\caption{Stable quasi-tube $_{s}Q$ of rank $5$}
\end{figure}
\medskip
\begin{figure}
\vspace{-1.5cm}
\begin{equation*}\label{AR-Gamma}
\begin{aligned}
\xymatrix@dr@R=25pt@C=25pt@!0{&&\scriptstyle\old{1}\ar[rr] &&\scriptstyle\tiny{\begin{matrix}3\\ \begin{matrix}1\end{matrix}~~\begin{matrix}4\\5\end{matrix} \end{matrix}} \ar[rr]& &\scriptstyle\tiny{\begin{matrix}  \begin{matrix}3\end{matrix}~~\\\begin{matrix}1\end{matrix}~~\begin{matrix}4\end{matrix}~~\end{matrix}}\ar[rr] && \scriptstyle\tiny{\begin{matrix}  \begin{matrix}4\end{matrix}~~\begin{matrix}3 \end{matrix}\begin{matrix}\ \end{matrix}\\ \begin{matrix}5\\3\end{matrix}~~\begin{matrix}4\end{matrix}~~\begin{matrix}1\end{matrix}\end{matrix}}
\ar[rr]&& \scriptstyle \tiny{\begin{matrix} \begin{matrix}\ \end{matrix}~~\begin{matrix}\ \end{matrix}~~\begin{matrix}\ \end{matrix}\begin{matrix}4\end{matrix}~~\begin{matrix}3 \end{matrix}\\~~\begin{matrix}2\end{matrix}~~\begin{matrix}5\end{matrix}~~\begin{matrix}4\end{matrix}~~\begin{matrix}1\end{matrix}\\\begin{matrix}3\end{matrix}~~\begin{matrix}\ \end{matrix}~~\begin{matrix}\ \end{matrix}~~\begin{matrix}\ \end{matrix}\end{matrix}}  \ar[rr] &&\scriptstyle\cdots\\	
&& \\	
&&\scriptstyle\tiny{\begin{matrix}1\\2\end{matrix}}\ar[uu] \ar[rr] &&\scriptstyle\new{\tiny{\begin{matrix}3\\ \begin{matrix}1\\2\end{matrix}~~\begin{matrix}4\\5\end{matrix} \end{matrix}}} \ar[rr]\ar[uu]&& 	\scriptstyle \old{\tiny{\Omega^{-1}(\begin{matrix}5\\3\end{matrix})=\begin{matrix}  \begin{matrix}3\end{matrix}~~\\\begin{matrix}1\\2\end{matrix}~~\begin{matrix}4\end{matrix}~~\end{matrix}} }\ar[rr]\ar[uu]&&\scriptstyle \tiny{\begin{matrix} \begin{matrix}\ \end{matrix}~~ \begin{matrix}4\end{matrix}~~\begin{matrix}3 \end{matrix}~~\begin{matrix}\ \end{matrix}\\ \begin{matrix}5\\3\end{matrix}~~\begin{matrix}4\end{matrix}~~\begin{matrix}1\\2\end{matrix}\end{matrix}} \ar[uu]\ar[rr]&& \scriptstyle \tiny{\begin{matrix} \begin{matrix}\ \end{matrix}~~\begin{matrix}\ \end{matrix}~~\begin{matrix}\ \end{matrix}\begin{matrix}4\end{matrix}~~\begin{matrix}3 \end{matrix}\\\begin{matrix}2\end{matrix}~~\begin{matrix}5\end{matrix}~~\begin{matrix}4\end{matrix}~~\begin{matrix}1\\2\end{matrix}\\\begin{matrix}3\end{matrix}~~\begin{matrix}\ \end{matrix}~~\begin{matrix}\ \end{matrix}~~\begin{matrix}\ \end{matrix}\end{matrix}} \ar[uu]\ar[rr] &&\scriptstyle \cdots\ar[uu] \\	
&& \\	
&&\scriptstyle \tiny{\begin{matrix} \begin{matrix}1\\2\end{matrix}~~\begin{matrix}4\\5\end{matrix}\\3 \end{matrix}} \ar[uu]\ar[rr] &&\scriptstyle \old{\tiny{\Omega(\begin{matrix}5\\3\end{matrix})=\begin{matrix}4\\5\end{matrix}}}\ar[rr]\ar[uu] &&\scriptstyle \new{4}\ar[rr]\ar[uu] && \scriptstyle\tiny{\begin{matrix} \begin{matrix} 4\end{matrix}\\\begin{matrix}5\\3\end{matrix}~~\begin{matrix}4\end{matrix}\end{matrix}}\ar[uu] \ar[rr]&& \scriptstyle \tiny{\begin{matrix} \begin{matrix}\ \end{matrix}~~\begin{matrix}\ \end{matrix}~~\begin{matrix}\ \end{matrix}\begin{matrix}4\end{matrix}\\~~\begin{matrix}2\end{matrix}~~\begin{matrix}5\end{matrix}~~\begin{matrix}4\end{matrix}\\\begin{matrix}3\end{matrix}~~~~\begin{matrix}\ \end{matrix}\end{matrix}}  \ar[uu]\ar[rr] &&\scriptstyle \cdots\ar[uu]\\	
&&	& \\
&&\scriptstyle \tiny{\cdots} \ar[rr] \ar[uu]&& \scriptstyle\tiny{\begin{matrix}\begin{matrix}\ \end{matrix}~~\begin{matrix}\ \end{matrix}~~\begin{matrix}\ \end{matrix}~~\begin{matrix}5\end{matrix} \\ \begin{matrix}\ \end{matrix}~~\begin{matrix}4\end{matrix}~~\begin{matrix}3\end{matrix}~~\\ \begin{matrix}5\end{matrix}~~\begin{matrix}4\end{matrix}\end{matrix}} \ar[rr]\ar[uu] &&\scriptstyle \tiny{\begin{matrix}\begin{matrix} 4\end{matrix}~~\begin{matrix}5\\3\end{matrix}\\\begin{matrix} 4\end{matrix}\end{matrix}}\ar[rr] \ar[uu] && \scriptstyle\tiny\old{\begin{matrix}5\\3\end{matrix}}\ar[uu] \ar[rr]&& \scriptstyle \tiny{\begin{matrix}\begin{matrix} 2\end{matrix}~~\begin{matrix}5\end{matrix}\\\begin{matrix} 3\end{matrix}\end{matrix}} \ar[uu]\ar[rr] &&\scriptstyle\tiny{\begin{matrix}1~~\begin{matrix}\ \end{matrix}~~\begin{matrix}\ \end{matrix}\\ \begin{matrix}1\end{matrix}~~\begin{matrix}\ \end{matrix}\begin{matrix}2\end{matrix}~~\begin{matrix}5\end{matrix}\\~~\begin{matrix}\ \end{matrix}~~\begin{matrix}\ \end{matrix}3 \end{matrix}} \ar[uu] \\
&&& \\
&&\scriptstyle\tiny{\cdots} \ar[uu] \ar[rr] && \scriptstyle\tiny{\begin{matrix} \begin{matrix}\ \end{matrix}~~\begin{matrix}4\end{matrix}~~\begin{matrix}3\end{matrix}~~\\ \begin{matrix}5\end{matrix}~~\begin{matrix}4\end{matrix}\end{matrix}}  \ar[rr]\ar[uu] && \scriptstyle\tiny{\begin{matrix}\begin{matrix} 4\end{matrix}~~\begin{matrix}3\end{matrix}\\\begin{matrix} 4\end{matrix}\end{matrix}}\ar[rr]\ar[uu] &&\scriptstyle\tiny{3} \ar[uu]\ar[rr]&& \scriptstyle \new{\tiny{\begin{matrix}2\\3\end{matrix}}}\ar[uu]\ar[rr] && \scriptstyle \old{\Omega^{-1}(1)=\tiny{\begin{matrix}1\\ \begin{matrix}1\end{matrix}~~\begin{matrix}2\\3\end{matrix} \end{matrix} }}\ar[uu]  \\ 
&&\\
&&\scriptstyle\tiny{\cdots}  \ar[uu] \ar[rr]&&\scriptstyle\tiny{\begin{matrix} \begin{matrix}1\end{matrix}~~\begin{matrix}3\end{matrix}~~\begin{matrix}4\end{matrix}\\ \begin{matrix}\ \end{matrix}~~\begin{matrix}\ \end{matrix}~~\begin{matrix}1 \end{matrix}~~\begin{matrix}4\end{matrix}~~\begin{matrix}5\end{matrix}\end{matrix}} \ar[rr]\ar[uu]&&\scriptstyle~~\tiny{\begin{matrix} \begin{matrix}1\end{matrix}~~\begin{matrix}3\end{matrix}~~\begin{matrix}4\end{matrix}\\ 1~~\begin{matrix}\ \end{matrix}~~\begin{matrix}4\end{matrix}\end{matrix}} \ar[rr]\ar[uu]&&\scriptstyle \tiny{\begin{matrix} \begin{matrix}1\end{matrix}~~\begin{matrix}3\end{matrix}\\1 \end{matrix}} \ar[uu]\ar[rr]&& \scriptstyle \old{\Omega(1)=\tiny{\begin{matrix} \begin{matrix}1\end{matrix}~~\begin{matrix}2\\3\end{matrix}\\1 \end{matrix}} }\ar[uu]\ar[rr] && \scriptstyle\old{1} \ar[uu] }
		\end{aligned}
\end{equation*}
\caption{Euclidean component $_{s}\Gamma$}
\end{figure}
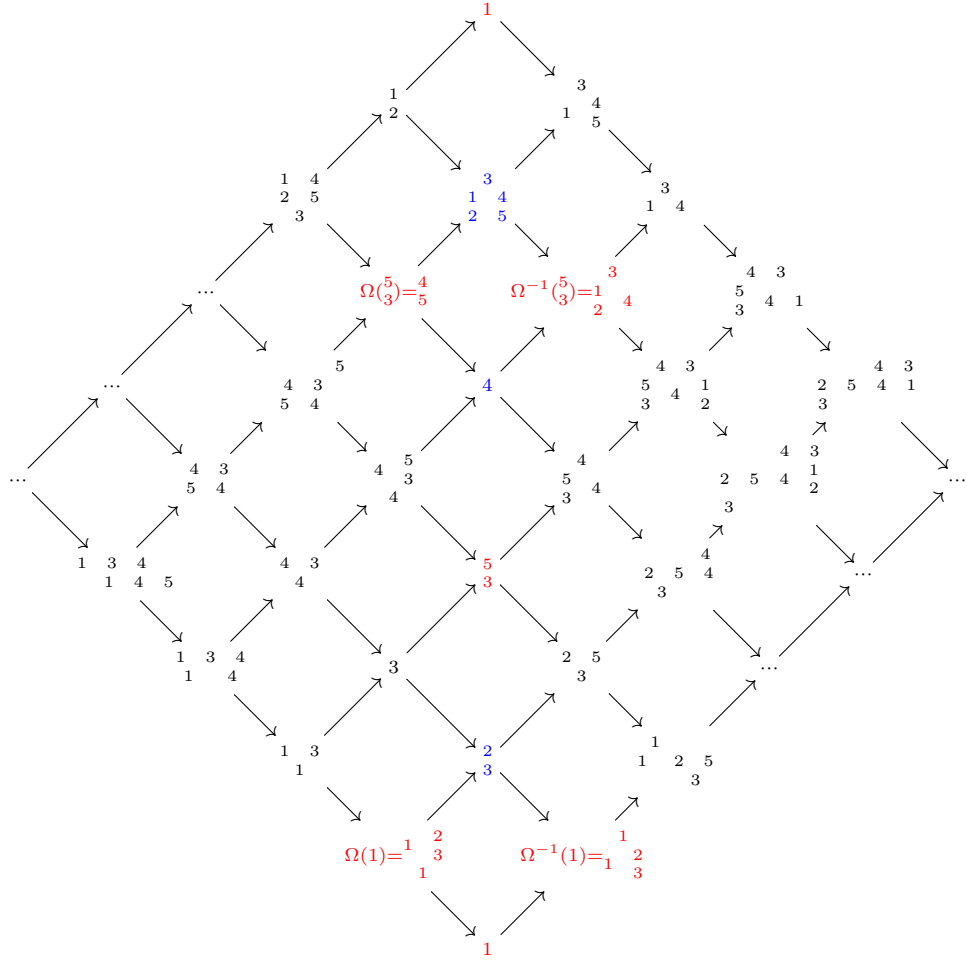
\end{Ex}

\section*{Declarations}

\subsection*{ Funding}
This work was supported by the National Natural Science Foundation of China (No. 12301044).

\subsection*{Data Availability}
No datasets were generated or analysed during the current study.

\subsection*{Ethical Approval} Not applicable.

\end{document}